\documentclass[a4paper]{article}
\usepackage{amssymb,amsmath,amsthm}
\usepackage{graphicx}
\usepackage{hyperref}
\usepackage{eucal}
\newcommand{\ie}{\emph{i.e.}}
\newcommand{\eg}{\emph{e.g.}}
\newcommand{\cf}{\emph{cf}}
\newcommand{\Nat}{\mathbb{N}}
\newcommand{\Real}{\mathbb{R}}

\newcommand{\sii}{L^2}
\newcommand{\Dom}{\mathfrak{D}}

\newcommand{\dist}{\mathop{\mathrm{dist}}\nolimits}
\newcommand{\supp}{\mathop{\mathrm{supp}}\nolimits}
\newcommand{\eps}{\varepsilon}
\newcommand{\const}{\mathrm{const}}
\newcommand{\esssup}{\mathop{\mathrm{ess\;\!sup}}}
\newcommand{\essinf}{\mathop{\mathrm{ess\;\!inf}}}
\newtheorem{Lemma}{Lemma}[section]
\newtheorem{Theorem}{Theorem}[section]
\newtheorem{Proposition}{Proposition}[section]
\newtheorem{Corollary}{Corollary}[section]
\theoremstyle{remark}
\newtheorem{Remark}{Remark}[section]
\theoremstyle{definition}
\newtheorem{Example}{Example}[section]
\numberwithin{equation}{section}
\begin{document}
%
%-------%
% TITLE %
%-------%
%------------------------------------------%
%------------------------------------------%
\title{\textbf{\LARGE
The Brownian traveller on manifolds
}}
\author{
Martin Kolb\,$^{a}$
\ and \
David Krej\v{c}i\v{r}{\'\i}k\,$^{b,c}$
}
\date{\small
%
%\begin{center}
\begin{quote}
\begin{enumerate}
\item[$a)$]
\emph{Department of Statistics, University of Oxford,
1 South Parks Road, Oxford OX1 3TG, United Kingdom};
kolb@stats.ox.ac.uk
\item[$b)$]
\emph{Department of Theoretical Physics,
Nuclear Physics Institute ASCR,
25068 \v Re\v z, Czech Republic};
krejcirik@ujf.cas.cz
%\medskip \\
\item[$c)$]
\emph{IKERBASQUE, Basque Foundation for Science,
48011 Bilbao, Kingdom of Spain}
\end{enumerate}
\end{quote}
%\end{center}
%
\ \smallskip \\
16 August 2011
}
\maketitle
\begin{abstract}
\noindent
We study the influence of the intrinsic curvature
on the large time behaviour of the heat equation
in a tubular neighbourhood of an unbounded geodesic
in a two-dimensional Riemannian manifold.
Since we consider killing boundary conditions,
there is always an exponential-type decay for the heat semigroup.
We show that this exponential-type decay is slower
for positively curved manifolds comparing to the flat case.
As the main result, we establish a sharp extra polynomial-type
decay for the heat semigroup on negatively curved manifolds
comparing to the flat case.
The proof employs the existence of Hardy-type inequalities
for the Dirichlet Laplacian in the tubular neighbourhoods
on negatively curved manifolds
and the method of self-similar variables
and weighted Sobolev spaces for the heat equation.
%
%\bigskip
%\begin{itemize}
%\item[\textbf{Keywords:}]
%\item[\textbf{MSC (2010):}]
%\end{itemize}
%
\end{abstract}
\newpage
%{\small
\tableofcontents
%}
%
%------------------------------------------%
%------------------------------------------%
%
\newpage
%---------------------%
\section{Introduction}
%---------------------%
%
The intimate intertwining between properties of Brownian motion
(or alternatively the heat flow) on a Riemannian manifold and
the curvature properties of the manifold are a classical research
question that has been investigated extensively
(see, \eg, \cite{Kendall1,Kendall2,Grig,GSL,Ne,Hsu})
and has led to deep results and new methods,
which turned out to be also of importance in other fields of mathematics.
One of the main themes here is to characterize probabilistic properties
via geometric ones and vice versa.
Thinking of the Brownian particle as a `traveller'
in a curved space we continue this line of research and investigate
the influence of the curvature on its large time behaviour.

However, in contrast to previous works,
we restrict the motion of the Brownian particle
to a tubular neighbourhood of a curve in the Riemannian manifold
and kill it when it leaves this quasi-one-dimensional subset.
This line of research seems to have its origin in the mathematical physics
literature, where one aims to describe the dynamics of quantum particles
in very thin almost one-dimensional waveguides.
The constraint on the Brownian motion to the quasi-one-dimensional
subsets leads to additional effects not present
in the case of an unrestricted stochastic conservative motion.
It particular it will turn out that the behaviour of the Brownian
particle in the tube-like set is sensitive to local perturbations
of the geometry.

A more precise description of our setting is the following.
Let the ambient space of the Brownian traveller
be a complete non-compact \emph{two-di\-men\-sion\-al}
Riemannian manifold~$\mathcal{A}$
(not necessarily embedded in the Euclidean space~$\Real^3$)
with Gauss curvature~$K$.
We restrict to the case of \emph{locally perturbed traveller}
by assuming that~$K$ is compactly supported.

We further assume that
the motion is \emph{quasi-one-dimensional}
in the sense that the Brownian traveller is forced to move
along an infinite curve~$\Gamma$ on the surface~$\mathcal{A}$.
To focus on the effects induced by the intrinsic curvature~$K$ itself,
we suppress side effects induced by the curvature of the curve
by assuming that~$\Gamma$ is a geodesic.

The constraint to move along the geodesic curve
is introduced by imposing \emph{killing boundary conditions}
on the boundary of the tubular neighbourhood
\begin{equation}\label{strip}
  \Omega := \left\{
  q\in\mathcal{A} \ | \ \dist(q,\Gamma) < a
  \right\}
  \,,
\end{equation}
where~$a$ is a positive (not necessarily small) number.
That is, the Brownian traveller `dies' whenever it hits
the boundary~$\partial\Omega$ of the strip~$\Omega$.

The problem is mathematically described by the diffusion equation
\begin{equation}\label{heat}
  \left\{
  \begin{aligned}
    \partial_t u -\Delta_q u &= 0
    &&\mbox{in} \quad \Omega \times (0,\infty)
    \,,
    \\
    u &= 0
    &&\mbox{on} \quad \partial\Omega \times (0,\infty)
    \,,
    \\
    u &= u_0
    &&\mbox{on} \quad \Omega \times \{0\}
    \,,
  \end{aligned}
  \right.
\end{equation}
in the space time variables $(q,t) \in \Omega \times (0,\infty)$,
where~$u_0$ is an initial datum.
More specifically, for the Dirac distribution $u_0(q)=\delta(q-q_0)$,
the solution $u(q,t)$~is related to the density
of the transition probability of the Brownian motion
starting at $q_0\in\Omega$ as follows.
Let us denote by $\mathbb{E}_q$ (respectively, $\mathbb{P}_q$)
the expectation (respectively, probability)
of Brownian motion $(X_t)_{t \geq 0}$ on the manifold $\mathcal{A}$
started at $q \in \mathcal{A}$ and
let $\tau_{\Omega} := \inf \lbrace t > 0 \mid X_t \in \partial \Omega\rbrace$
denote the first exit time. Then
\begin{equation}\label{e:expheat}
u(q,t)= \mathbb{E}_q \bigl[ u_0(X_t),\tau_{\Omega}>t\bigr]
\end{equation}
solves equation \eqref{heat}.
If $u_0 = \chi_B$ for some measurable set $B\subset \Omega$, we get
\begin{equation}\label{e:probheat}
u(q,t) = \mathbb{P}_q\bigl(X_t \in B,\tau_{\Omega}>t\bigr)
\,,
\end{equation}
which is
the probability that the Brownian particle survived up to time $t$
and is in~$B$ at time $t$.

Now imagine a Brownian traveller in~$\Omega$
and we imagine that he/she reached his/her goal when hitting the boundary.
The ultimate question we would like to address in this paper
is to decide \emph{which geometry is better to travel}.
By the `good geometry' we understand that which
enables the Brownian traveller to reach his/her goal
as soon as possible or `to escape from his/her starting point as far as possible'.
More precisely, we are interested in quantifying
the large time of \eqref{e:probheat} for bounded sets $B \subset \Omega_0$.

In any case, the question is related to the large time decay
of the solutions of~\eqref{heat} as regards the curvature~$K$.
We mainly study a Hilbert-space version of the problem
by analysing the asymptotic behaviour
of the heat semigroup on $\sii(\Omega)$ associated with~\eqref{heat}.
Nevertheless, we establish some pointwise results
about the large time behaviour of~$u(q,t)$ as well.

Our results are informally summarized in Table~\ref{table}.
\begin{table}[!h]
\begin{center}
\begin{tabular}{c||c|c|c}
curvature & positive & zero & negative
\\ \hline
transport & \emph{bad} & \emph{critical} & \emph{good}
\\ \hline
probability decay
& $e^{\gamma t} \, e^{-E_1 t}$
& $t^{-1/2} \, e^{-E_1 t}$
& $t^{-3/2} \, e^{-E_1 t}$
\rule{0ex}{2.5ex}
\end{tabular}
\end{center}
\caption{An informal summary of our results.}\label{table}
\end{table}

\noindent
There $E_1:=\pi^2/(2a)^2$ denotes the lowest Dirichlet eigenvalue
of the strip cross-section $(-a,a)$ and~$\gamma$ is a positive number.
As explained above, the vague statements about transport
in Table~\ref{table} should be understood in the spirit
of the large time decay of the solutions to~\eqref{heat} stated there.
It turns out that the solutions of~\eqref{heat}
has worse (respectively, better) decay properties
if~$K$ is non-negative (respectively, non-positive)
as a consequence of the existence of stationary solutions
(respectively, Hardy-type inequalities).
More general results, involving surfaces with sign changing curvatures,
are established in this paper.

The effect of curvature on the transience/recurrence of a Brownian particle
have been extensively studied (see~\cite{Grig} for a nice review).
It turned out that on manifolds with `large'
negative curvature Brownian motion leaves compact subsets
faster than on manifolds with non-negative curvature.
But local changes of the Riemannian metric cannot change transience
to recurrence or vice versa.
Observe that for the results presented in Table~\ref{table}
this is non longer true.
In probabilistic literature this  corresponds
to the $R$-recurrence/$R$-transience dichotomy (see \cite{T74a,T74b})
or in analytic literature to the
critical/subcritical dichotomy (see, \eg, \cite{Pin},
or \cite{Pinchover_2004,Pinchover_2007} for a brief overview).
Indeed, in our setting the Brownian motion
in the negatively curved tube with \textit{compactly supported}
curvature is $E_1$-transient in contrast to the case of no curvature.

The organization of this paper is as follows.
In the forthcoming Sections~\ref{Sec.Geometry} and~\ref{Sec.Analytic}
we properly define the configuration space of the Brownian traveller
and the associated heat equation~\eqref{heat}, respectively.
The case of zero curvature is briefly mentioned in Section~\ref{Sec.zero}.
In Section~\ref{Sec.vanish} we consider direct consequences
in a more general situation when the curvature vanishes at infinity.
The influence of positive curvature on the Brownian traveller
is studied in Section~\ref{Sec.positive}.
The main part of the paper consists of Section~\ref{Sec.negative},
where we establish the existence of Hardy-type inequalities
in negatively curved manifolds
and develop the method of self-similar variables
for the heat equation to reveal the subtle effect
of negative curvature.
The paper is concluded by Section~\ref{Sec.end}
where we summarize our results
and refer to some open problems.

%--------------------------------%
\section{Geometric preliminaries}\label{Sec.Geometry}
%--------------------------------%
%
We start by imposing some natural hypotheses to give
an instructive geometrical interpretation of
the configuration space~$\Omega$ of the Brownian traveller.
The conditions will be considerably weakened later
when we reconsider the problem in an abstract setting.

%-----------------------------------%
\subsection{The configuration space}
%-----------------------------------%
%
Let us assume that the Riemannian manifold~$\mathcal{A}$
is of class~$C^2$ and that its Gauss curvature~$K$ is continuous.
The latter holds under the additional assumption
that~$\mathcal{A}$ is either of class~$C^3$
(by Gauss's Theorema Egregium)
or that it is embedded in~$\Real^3$
(by computing principal curvatures).

Any geodesic curve~$\Gamma$ on~$\mathcal{A}$ is $C^2$-smooth
and, without loss of generality, we may assume that it is
parameterized by arc-length.
To enable the traveller to propagate to infinity,
we consider unbounded geodesics~$\Gamma$ only.
For a moment,
we make the strong hypothesis that $\Gamma:\Real\to\mathcal{A}$
is an embedding.
%Let us remark that
%we do not need to assume that~$\mathcal{A}$ is complete
%to have an infinite geodesic~$\Gamma$,
%we only need to ensure that our~$\Gamma$ does not
%hit or approach the boundary of~$\mathcal{A}$ if present.

Since~$\Gamma$ is parameterized by arc-length,
the derivative $T:=\dot\Gamma$ defines
the unit tangent vector field along~$\Gamma$.
Let~$N$ be the unit normal vector field along~$\Gamma$
which is uniquely determined as the $C^1$-smooth mapping
from~$\Real$ to the tangent bundle of~$\mathcal{A}$
by requiring that~$N(s)$ is orthogonal to~$T(s)$
and that $\{T(s),N(s)\}$ is positively oriented for all $s\in\Real$
(\cf~\cite[Sec.~7.B]{Spivak4}).

The feature of our model is that the Brownian traveller is assumed
to be confined to the strip-like $a$-tubular neighbourhood~\eqref{strip}.
By definition, $\Omega$~is the set of points~$q$ in~$\mathcal{A}$
for which there exists a geodesic of length less than~$a$
from~$q$ meeting~$\Gamma$ orthogonally.
More precisely, we introduce a mapping~$\mathcal{L}$
from the flat strip
\begin{equation}\label{Omega0}
 \Omega_0 := \Real\times(-a,a)
\end{equation}
(considered as a subset of the tangent bundle of~$\mathcal{A}$)
to the manifold~$\mathcal{A}$ by setting
\begin{equation}\label{exp}
  \mathcal{L}(x_1,x_2) := \exp_{\Gamma(x_1)}(N(x_1) \,x_2)
  \,,
\end{equation}
where $\exp_q$ is the exponential map of~$\mathcal{A}$ at $q\in\mathcal{A}$.
Then we have
\begin{equation}\label{image}
  \Omega = \mathcal{L}(\Omega_0)
  \,.
\end{equation}
Note that $x_1\mapsto\mathcal{L}(x_1,x_2)$ traces the curves
parallel to~$\Gamma$ at a fixed distance~$|x_2|$,
while the curve~$x_2\mapsto\mathcal{L}(x_1,x_2)$
is a geodesic orthogonal to~$\Gamma$ for any fixed~$x_1$.
See Figure~\ref{Fig.Fermi}.

\begin{figure}[h!]
\begin{center}
\includegraphics[width=0.7\textwidth]{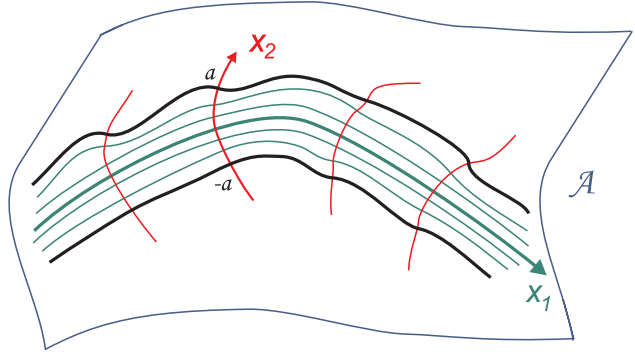}
\caption{The parametrization of the strip~$\Omega$ via the Fermi
coordinates $x=(x_1,x_2)$ defined by~\eqref{exp}.}\label{Fig.Fermi}
\end{center}
\end{figure}
%

%---------------------------------%
\subsection{The Fermi coordinates}
%---------------------------------%
%
Making the hypothesis that
\begin{equation}\label{Ass.basic}
  \mbox{$\mathcal{L} : \Omega_0 \to \Omega$ is a diffeomorphism}
  \;\!,
\end{equation}
we get a convenient parametrization of~$\Omega$
via the (Fermi or geodesic parallel) `coordinates'
$x=(x_1,x_2)$
determined by~\eqref{exp}, \cf~Figure~\ref{Fig.Fermi}.
We refer to~\cite[Sec.~2]{Gray} and~\cite{Hartman_1964}
for the notion and properties of Fermi coordinates.
In particular, it follows by the generalized Gauss lemma
that the metric~$G$ induced by~\eqref{exp}
%\ie\ $G_{ij}:=\partial_i\mathcal{L} \cdot \partial_j\mathcal{L}$,
acquires the diagonal form:
\begin{equation}\label{metric}
  G =
  \begin{pmatrix}
    f^2 & 0 \\
    0   & 1
  \end{pmatrix}
  \,,
\end{equation}
where~$f$ is continuous and has continuous partial derivatives
$\partial_2 f$, $\partial_2^2 f$ satisfying the Jacobi equation
\begin{equation}\label{Jacobi}
  \partial_2^2 f + \,K f = 0
  \qquad\textrm{with}\qquad\left\{
  \begin{aligned}
    f(\cdot,0) &= 1 \,, \\
    \partial_2 f(\cdot,0) &= 0 \,.
  \end{aligned}
  \right.
\end{equation}
Here~$K$ is considered as a function of the Fermi coordinates~$(x_1,x_2)$.

By the inverse function theorem, a sufficient condition
to ensure~\eqref{Ass.basic} is that~$\mathcal{L}$ is injective
and~$f$ positive. The latter can always be achieved
for sufficiently small~$a$ as the following lemma shows.
\begin{Lemma}\label{Lem.Taylor}
Let $K \in L^\infty(\Omega_0)$ and $\|K\|_\infty \;\! a^2 < 1$.
For every $x\in\Omega_0$, we have
\begin{equation}\label{bound.Taylor}
  1 - \frac{\bar{K}(x_1) \;\! a^2}{1-\bar{K}(x_1) \;\! a^2}
  \ \leq\  f(x) \ \leq\
  1 + \frac{\bar{K}(x_1) \;\! a^2}{1-\bar{K}(x_1) \;\! a^2}
  \,,
\end{equation}
where $\|\cdot\|_\infty := \|\cdot\|_{L^\infty(\Omega_0)}$ and
$$
  \bar{K}(x_1) := \esssup_{x_2\in(-a,a)} |K(x_1,x_2)|
  \,.
$$
\end{Lemma}
\begin{proof}
Integrating~\eqref{Jacobi}, we arrive at the identity
$$
  \forall x\in\Omega_0 \,, \qquad
  \partial_2 f(x) = - \int_0^{x_2} (K f)(x_1,\xi) \, d\xi
  \,.
$$
Consequently,
\begin{equation}\label{f2.Taylor}
  |\partial_2 f(x)| \leq a \, \bar{K}(x_1) \;\! \bar{f}(x_1)
  \,,
  \qquad \mbox{with} \qquad
  \bar{f}(x_1) := \sup_{\xi\in(-a,a)} |f(x_1,\xi)|
  \,,
\end{equation}
for all $x\in\Omega_0$.
By the mean value theorem, we deduce the bounds
\begin{equation}\label{bound.Taylor.bis}
  \forall x\in\Omega_0 \,, \qquad
  1 - a^2 \, \bar{K}(x_1) \, \bar{f}(x_1)
  \leq f(x) \leq
  1 + a^2 \, \bar{K}(x_1) \, \bar{f}(x_1)
  \,.
\end{equation}
Taking the supremum over $x_2\in(-a,a)$,
the upper bound leads to the upper bound of~\eqref{bound.Taylor}.
Finally, using the upper bound of~\eqref{bound.Taylor}
to estimate~$\bar{f}$ in the lower bound of~\eqref{bound.Taylor.bis},
we conclude with the lower bound of~\eqref{bound.Taylor}.
\end{proof}
%

%---------------------------------%
\subsection{The abstract setting}
%---------------------------------%
%
It follows from the preceding subsection that,
under the hypothesis~\eqref{Ass.basic},
we can identify $\Omega \subset \mathcal{A}$
with the Riemannian manifold $(\Omega_0,G)$.
However, the assumption~\eqref{Ass.basic}
is not really essential provided that one is ready to abandon
the geometrical interpretation of~$\Omega$
as a tubular neighbourhood embedded in~$\mathcal{A}$.

Indeed, $(\Omega_0,G)$, with the metric~$G$ determined
by~\eqref{metric} and~\eqref{Jacobi},
can be considered as an abstract Riemannian manifold
for which the boundedness of~$K$ and a restriction of~$a$
are the only important hypotheses.
More specifically, we assume
\begin{equation}\label{Ass.basic.alt}
  K \in L^\infty(\Omega_0)
  \qquad\mbox{and}\qquad
  \|K\|_\infty \;\! a^2 < \frac{1}{2}
  \,.
\end{equation}
Then the Jacobi equation~\eqref{Jacobi}
admits a solution $f(x_1,\cdot) \in H^2((-a,a))$
for every $x_1\in\Real$ and it follows from Lemma~\ref{Lem.Taylor}
that~$f$ is bounded and uniformly positive on~$\Omega_0$.

In the sequel, we therefore allow for self-intersections
and low regularity of~$\Omega$ by considering $(\Omega_0,G)$
as an abstract configuration space of the Brownian traveller.
The mere boundedness of the metric~$G$ is sufficient
to establish the desired results.

%-------------------------------------------------%
\section{Analytic and probabilistic preliminaries}\label{Sec.Analytic}
%-------------------------------------------------%
%
In this section, we give a precise meaning to
the evolution problem~\eqref{heat}.
%and state some rather elementary properties of the solutions.

%-----------------------------------%
\subsection{The generator of motion}
%-----------------------------------%
%
The meaning of~$-\Delta_q u$ in~\eqref{heat}
should be understood as an action of
the Laplace-Beltrami operator~$-\Delta$
in the Riemannian manifold~$\Omega$.
In the Fermi coordinates,
considering~$-\Delta$ as a differential expression in~$\Omega_0$,
we have
\begin{equation}\label{LB}
  -\Delta = - |G|^{-1/2} \partial_i |G|^{1/2} G^{ij} \partial_j
  = - f^{-1} \partial_1 f^{-1} \partial_1
  - f^{-1} \partial_2 f \partial_2
  \,.
\end{equation}
Here the first identity is a general formula for the Laplace-Beltrami
operator in a manifold equipped with the metric~$G$,
with the usual notation for the determinant $|G|:=\det(G)$
and the coefficients~$G^{ij}$ of the inverse metric~$G^{-1}$,
and using the Einstein summation convention.
The second identity employs the special form of the metric~\eqref{metric}
in the Fermi coordinates.

The objective of this subsection is to associate
to the differential expression~\eqref{LB} a self-adjoint operator~$H_K$
in the Hilbert space
\begin{equation}\label{Hilbert}
  \sii_f(\Omega_0) :=
  \sii\big(\Omega_0,f(x)\;\!dx\big)
  \,,
\end{equation}
a space isomorphic to $\sii(\Omega)$ via the Fermi coordinates.
In order to implement the Dirichlet boundary conditions of~\eqref{heat},
we introduce~$H_K$ as the Friedrichs extension of~\eqref{LB}
initially defined on smooth functions of compact support in~$\Omega_0$
(\cf~\cite[Sec.~6]{Davies}).
That is, $H_K$~is the unique self-adjoint operator associated
on~\eqref{Hilbert} with the quadratic form
\begin{equation}\label{form}
  h_K[\psi]
  := \big(\partial_i\psi,G^{ij} \partial_j\psi\big)_{f}
  \,, \qquad
  \psi\in\Dom(h_K)
  := H_0^1(\Omega_0,G)
  \,.
\end{equation}
Here $(\cdot,\cdot)_f$ denotes the inner product in~\eqref{Hilbert}
and $H_0^1(\Omega_0,G)$ denotes the completion
of $C_0^\infty(\Omega_0)$ with respect to the norm
$
  \|\cdot\|_{\Dom(h_K)} :=
  (
  h_K[\cdot] + \|\cdot\|_f^2
  )^{1/2}
$,
with~$\|\cdot\|_f$ denoting the norm in~\eqref{Hilbert}.
The dependence of~$H_K$ on the curvature~$K$
is understood through the dependence of~$f$ on~$K$,
\cf~\eqref{Jacobi}.

Under our hypothesis~\eqref{Ass.basic.alt},
it follows from Lemma~\ref{Lem.Taylor} that
$\|\cdot\|_f$~is equivalent to the usual
norm $\|\cdot\|:=\|\cdot\|_1$
in $\sii(\Omega_0)=\sii_1(\Omega_0)$ (\ie~$f=1$)
and, moreover, the $\Dom(h_K)$-norm is equivalent to the usual norm
in the Sobolev space $H^1(\Omega_0)$.
Consequently,
$$
  \Dom(h_K)=H_0^1(\Omega_0)
  \,.
$$
However, it is important to keep in mind that,
although $H^1(\Omega_0,G)$ and $H_0^1(\Omega_0)$
coincide as vector spaces, their topologies are different.

\begin{Remark}
Under extra regularity assumptions involving derivatives of~$f$,
it is possible to show that~$H_K$ acts as~\eqref{LB}
on the domain $H_0^1(\Omega_0) \cap H^2(\Omega_0)$.
However, we shall not need these facts,
always considering~$H_K$ in the form sense described above.
\end{Remark}
%

%------------------------%
\subsection{The dynamics}\label{Sec.dynamics}
%------------------------%
%
As usual, we consider the weak formulation of
the parabolic problem~\eqref{heat}.
We say a Hilbert space-valued function
$
  u \in \sii_\mathrm{loc}\big((0,\infty);H_0^1(\Omega_0,G)\big)
$,
with the weak derivative
$
  u' \in \sii_\mathrm{loc}\big((0,\infty);[H_0^{1}(\Omega_0,G)]^*\big)
$,
is a (global) solution of~\eqref{heat} provided that
\begin{equation}\label{heat.weak}
  \big\langle v,u'(t)\big\rangle_{\!f} + h_K\big(v,u(t)\big) = 0
\end{equation}
for each $v \in H_0^1(\Omega_0,G)$ and a.e.\ $t\in[0,\infty)$,
and $u(0)=u_0$.
Here $h_K(\cdot,\cdot)$ denotes the sesquilinear form
associated with~\eqref{form}
and $\langle\cdot,\cdot\rangle_f$
stands for the pairing of $H_0^1(\Omega_0,G)$
and its dual $[H_0^{1}(\Omega_0,G)]^*$.
With an abuse of notation, we denote by the same symbol~$u$
both the function on $\Omega_0\times(0,\infty)$
and the mapping $(0,\infty) \to H_0^1(\Omega_0,G)$.

Standard semigroup theory implies that there indeed exists
a unique solution of~\eqref{heat.weak} that belongs to
$C^0\big([0,\infty);\sii_f(\Omega_0)\big)$.
More precisely, the solution is given by $u(t) = e^{-t H_K} u_0$,
where~$e^{-t H_K}$ is the semigroup associated with~$H_K$.

It is easy to see that the real and imaginary parts
of the solution~$u$ of~\eqref{heat} evolve separately.
By writing $u = \Re(u) + i \, \Im(u)$ and solving~\eqref{heat}
with initial data $\Re(u_0)$ and $\Im(u_0)$,
we may therefore reduce the problem to the case of a real function~$u_0$,
without restriction.
This reflects the fact that~$e^{-t H_K}$ is positivity preserving.
Consequently, the functional spaces can be considered to be real
when investigating the heat equation~\eqref{heat}.

Indeed, the quadratic form $h_K$ is a Dirichlet form,
to which we can associate a strong Markov process
with continuous paths (Brownian motion on $(\Omega_0,G)$).
In order to do so let us first extend $f$ to $\mathbb{R}^2$
by setting it equal to $1$ outside $\Omega_0$.
Moreover, let us define the Dirichlet form $\tilde{h}_K$
in $L^2(\mathbb{R}^2, f(x)\;\!dx)$ by
\begin{displaymath}
  \tilde{h}_K[\psi]
  := \int_{\mathbb{R}^2}
  \overline{\partial_i\psi(x)} \;\! G^{ij}(x) \;\! \partial_j\psi(x)
  \,f(x)\,dx
  \,, \qquad
  \psi\in\Dom(\tilde{h}_K) := H^1(\mathbb{R}^2)
  \,.
\end{displaymath}
Then there exists a strong Markov process $(X_t)_{t \geq 0}$
with continuous paths, which is associated to $\tilde{h}_K$.
According to Theorem 4 in~\cite{StI} the process is conservative.
We use $\mathbb{E}_x$ (respectively, $\mathbb{P}_x$)
to denote the expectation (respectively, probability) conditional on $X_0 = x$.
Since Dirichlet boundary conditions correspond to killing
in the probabilistic picture, we have the following probabilistic
representation
\begin{equation}\label{e:probdirichletform}
e^{-tH_K}u_0 (x)=\mathbb{E}_x\bigl[u_0(X_t),\tau_{\Omega_0}>t\bigr]
\end{equation}
for almost every $x \in \Omega_0$.

%----------------------------%
\subsection{Basic properties}
%----------------------------%
%
In our first proposition we collect some fundamental properties
of the stochastic process $(X_t)_{t\geq 0}$.
\begin{Proposition}\label{p:fundprop}
Assume~\eqref{Ass.basic.alt}.
\itemize
\item The stochastic process $(X_t)_{t \geq 0}$ has the strong Feller property
and is therefore well-defined for every $x \in \Omega_0$.
In particular, the right hand side of~\eqref{e:probdirichletform}
is continuous for every $u_0 \in L^{\infty}(\Omega_0)$.
\item The stochastic process $(X_t)_{t \geq 0}$ has a continuous
transition function $k_t(\cdot,\cdot)$ with respect to $f(x)\;\!dx$,
which satisfy a Gaussian bound,
\ie\ for some constants $C_1>0$, $C_2>0$, one has
\begin{displaymath}
\forall x,y \in \Omega_0, \qquad
k_t(x,y) \leq \frac{C_1}{t} \, e^{\frac{|x-y|^2}{C_2t}}
\,.
\end{displaymath}
\end{Proposition}
\begin{proof}
The first assertion follows immediately from the second one
by a standard use of Lebesgue's dominated convergence theorem.

In order to prove the second assertion,
let us denote by $\tilde{H}_K$ the unique self-adjoint operator
associated to $\tilde{h}_K$.
Observe that according to \cite[Thm.~1.1]{N}
the semigroup $e^{-t\tilde{H}_K}$ has an integral kernel,
satisfying a Gaussian upper bound.
As $e^{tH_K}$ is dominated by $e^{-t\tilde{H}_K}$
(using either \cite{Ouh} or the probabilistic representation),
this bound for $\tilde{H}_K$ carries over to $H_K$.
In order to prove the regularity assertion concerning
the transition kernel,
observe that the Dirichlet form $\tilde{h}_K$ corresponds
to a uniformly elliptic operator
(in the sense of \cite[Sec.~4]{LSC})
on the subset $\Omega_0$ of the Riemannian manifold
$\mathbb{R}^2$ with Euclidean metric.
Thus, according to the remark below Theorem 6.3 in \cite{LSC}
(compare also \cite{StII}),
it therefore follows that the transition kernel
is locally H\"older continuous.
\end{proof}

In this work we are mainly interested in the large time behaviour
of the stochastic process $(X_t)_{\tau_{\Omega_0} > t \geq 0}$,
which is well-known to be connected to spectral properties
of its generator~$H_K$. The spectral mapping theorem yields
\begin{equation}\label{SMT}
  \big\|e^{-t H_K}\big\|_{\sii_f(\Omega_0)\to\sii_f(\Omega_0)}
  = e^{-\lambda_K t}
\end{equation}
for each time $t \geq 0$, where~$\lambda_K$ denotes
the lowest point in the spectrum of~$H_K$, \ie,
$
  \lambda_K := \inf\sigma(H_K)
$.
Hence, it is important to understand the low-energy
properties of~$H_K$ in order to study the large time
behaviour of the solutions of~\eqref{heat}.

From equation \eqref{SMT} and Proposition \ref{p:fundprop}
we deduce the following result showing that the exponential rate
of decay of $\mathbb{P}_x\bigl(X_t \in B,\tau_{\Omega_0}>t\bigr)$
is given by the lowest point in the spectrum.
\begin{Proposition}\label{p:pointexpon}
Assume~\eqref{Ass.basic.alt}.
For any measurable subset $B\subset \Omega_0$
and every $x \in \Omega_0$,
\begin{displaymath}
-\lim_{t \rightarrow \infty}\frac{1}{t}
\log\mathbb{P}_x\bigl(X_t \in B,\tau_{\Omega_0}>t\bigr)
=\lambda_K.
\end{displaymath}
\end{Proposition}
\begin{proof}
We apply arguments from \cite{Sim2} and \cite{Sim3}
used there in the context of Schr\"odinger operators.
First observe that the positive Sub-Markov operators
$e^{-tH_K}$ act as bounded operators on the space
$L^{\infty}(\Omega_0,f(x)\;\!dx)$ and by duality also
on $L^1(\Omega_0,f(x)\;\!dx)$. Let us set
\begin{displaymath}
\alpha_p :=
\lim_{t \rightarrow \infty}\frac{1}{t}
\log\bigl\|e^{-tH_K}\bigr\|_{L_f^p(\Omega_0)\rightarrow L_f^p(\Omega_0)}
,
\end{displaymath}
with the notation $L_f^p(\Omega_0):=L^p(\Omega_0,f(x)\;\!dx)$.
Then we have $\alpha_p=\alpha_{p'}$ ($p^{-1}+p'^{-1}=1$)
and $\alpha_p \leq \alpha_q$ $(2\leq p \leq q$).
On the other hand, using  the Gaussian bound in Proposition~\ref{p:fundprop},
we get for $\psi \in L^2_f(\Omega_0)$, $t > 2$ and some constant $C>0$,
\begin{equation*}
\begin{split}
\bigl\|e^{-tH_K}\psi\bigr\|_{L_f^{\infty}(\Omega_0)}
\leq C\,\bigl\|e^{-tH_K}\psi\bigr\|_{L^{2}_f(\Omega_0)}
\leq C\,e^{-\lambda_K(t-1)}\|\psi\|_{L^2_f(\Omega_0)}
\,.
\end{split}
\end{equation*}
Let $\psi$ denote the indicator function of the set $\Omega_0\cap B(0,r)$,
where $B(0,r)$ denotes the ball with radius~$r$ centered at~$0$.
Then we get for $x \in \Omega_0$
\begin{displaymath}
\mathbb{P}_x\bigl(X_t \in \Omega_0\cap B(0,r), \tau_{\Omega_0}>t\bigr)
= e^{-tH_K}\psi(x)
\leq C \, e^{\lambda_K}\sqrt{r}\,e^{-\lambda_Kt}.
\end{displaymath}
On the other hand we have (see also \cite[p.~429]{LSC})
for some $C_1' >0$
\begin{displaymath}
\mathbb{P}_x\bigl(X_t \in \Omega_0\cap B(0,r), \tau_{\Omega_0}>t\bigr)
= \int_{\Omega_0\cap B(0,r) }k_t(x,y)f(y)\,dy\leq e^{r^2/C_1't}.
\end{displaymath}
Choosing $r=\rho t$ with sufficiently large~$\rho$,
this finishes the proof of the upper bound
\begin{displaymath}
\limsup_{t \rightarrow \infty}\frac{1}{t}
\log\mathbb{P}_x\bigl(X_t \in B,\tau_{\Omega_0}>t\bigr) \leq \lambda_K.
\end{displaymath}
In order to proof the assertion of the Lemma we follow
the proof of Theorem~A.1.2. in \cite{Sim2}.
It is sufficient to prove that for every $\varepsilon > 0$
there exists a constant $c>0$ such that for sufficiently large $t>0$
\begin{displaymath}
\mathbb{P}_x\bigl(X_t \in B,\tau_{\Omega_0}>t\bigr)
=e^{-\lambda_K t}\chi_B (x) \geq c \;\! e^{-(\lambda_K+\varepsilon)t}.
\end{displaymath}
We set ${H}_K' := H_K-\lambda_K$.
There exists a smooth compactly supported
$\psi \in H^1_0(\Omega_0)$ with $\|\psi\|_{f}=1$
such that $h_K[\psi]-\lambda_K<\varepsilon/2$.
Let $W$ be $-\varepsilon$ on some bounded ball
containing the support of $\psi$ and $0$ otherwise.
Then the operators $H_K'-W$ and $H_K'$
have the same essential spectrum.
>From the inequality
$(\psi,(H_K'-W)\psi)_{L^2_f(\Omega_0)} < -\varepsilon/2$,
we conclude that the bottom of the spectrum $\lambda_{K,W}$
of $H_K'-W$ is a negative isolated eigenvalue
and the associated ground state~$\psi_0$ can be chosen to be non-negative.
Since $e^{-t(H_K'-W)} \leq e^{\varepsilon t} e^{-tH_K'}$,
we then arrive at ($t >1$)
\begin{equation*}
\begin{split}
e^{\lambda_{K,W} t}e^{-t_0(H_K-W)}\chi_B(x)
&= e^{\lambda_{K,W} t}
\big(
e^{-(H_K-W)1}(x,\cdot),e^{-(t-1)H_K'}\chi_B
\big)_{f} \\
&\xrightarrow[t \to \infty]{} \psi_0(x)\bigl(\psi_0,\chi_B)_{f}
\,,
\end{split}
\end{equation*}
and therefore at $C e^{-\varepsilon t}  \leq  e^{-tH_K'}\chi_B(x)$
for some constant $c>0$.
\end{proof}

A better understanding of low-energy properties of~$H_K$
leads to much more precise estimates.

%-----------------------%
\section{Flat manifolds}\label{Sec.zero}
%-----------------------%
%
We say that (a submanifold of) $\mathcal{A}$ is \emph{flat}
if its Gauss curvature~$K$ is identically equal to zero (on the submanifold).
The Brownian motion in a flat ambient space is easy to understand
because~$\Omega$ coincides with the straight Euclidean strip~$\Omega_0$,
\ie~$G$ is identity, for which the heat equation~\eqref{heat}
can be solved by separation of variables.

%-----------------------------------%
\subsection{Separation of variables}
%-----------------------------------%
%
By the `separation of variables' mentioned above we mean precisely
that the Dirichlet Laplacian $H_0 = -\Delta_D^{\Omega_0}$
on $\sii(\Omega_0)$ can be identified with the decomposed operator
\begin{equation}\label{decomposition}
  (-\Delta^\Real) \otimes 1 + 1 \otimes
  (-\Delta_{D}^{(-a,a)})
  \qquad\mbox{in}\qquad
  \sii(\Real) \otimes \sii((-a,a))
  \,.
\end{equation}
Here we denote by $-\Delta_D^U$ the Dirichlet Laplacian on $\sii(U)$
for any open Euclidean set~$U$,
suppress the subscript~$D$ if the boundary of~$U$ is empty,
and~$1$ stands for the identity operators in the appropriate spaces.
In a probabilistic language,
\eqref{decomposition}~is essentially a reformulation of the fact
that the horizontal and the vertical component of $(X_t)_{t\geq 0}$
are independent.

The eigenvalues and (normalized) eigenfunctions of~$-\Delta_{D}^{(-a,a)}$
are respectively given by ($n = 1,2,\dots$)
\begin{equation}\label{spectrum.straight}
  E_n := \left(\frac{n\pi}{2 a}\right)^2
  \,, \qquad
  \mathcal{J}_{n}(x_2)
  := \sqrt{\frac{1}{a}} \, \sin\left[ E_n (x_2+a) \right]
  \,,
\end{equation}
while the spectral resolution of~$-\Delta^\Real$
is obtained by the Fourier transform.
Then it is easy to see that the heat semigroup~$e^{-t H_0}$
is an integral operator with kernel
\begin{equation}\label{kernel.straight}
  s_0(x,x',t) := \sum_{n=1}^\infty e^{-E_n t} \,
  \mathcal{J}_n(x_2) \, p(x_1,x_1',t) \, \mathcal{J}_n(x_2')
  \,,
\end{equation}
where
$$
  p(x_1,x_1',t) := \frac{e^{-(x_1-x_1')^2/(4t)}}{\sqrt{4\pi t}}
$$
is the well known heat kernel of~$-\Delta^\Real$.

%--------------------------%
\subsection{The decay rate}
%--------------------------%
%
Concerning the large time behaviour of~$e^{-t H_0}$,
it follows from the decomposition~\eqref{decomposition} that
\begin{equation}\label{spectrum.straight.full}
  \sigma(H_0) = \sigma_\mathrm{ess}(H_0)
  = [E_1,\infty)
  \,,
\end{equation}
and therefore, as a consequence of~\eqref{SMT},
\begin{equation}\label{SMT.flat}
  \big\|e^{-t H_0}\big\|_{\sii(\Omega_0)\to\sii(\Omega_0)}
  = e^{-E_1 t}
\end{equation}
for each time $t \geq 0$.
Consequently, any solution of~\eqref{heat} satisfies
the global decay estimate
$
  \|u(t)\|
  \leq e^{-E_1 t} \;\! \|u_0\|
$
for every $t \geq 0$.

However, it is possible to obtain an extra polynomial decay
for solutions with initial data decaying sufficiently
fast at the infinity of the strip~$\Omega_0$.
To see it, let us consider the weight function
\begin{equation}\label{weight}
  w(x) := e^{x_1^2/4}
\end{equation}
and restrict the class of initial data to those~$u_0$
which belong to the weighted space
$
  \sii_w(\Omega_0)
$
defined in the same way as~\eqref{Hilbert}.
Then we have the improved decay estimate
$
  \|u(t)\|
  \leq C \;\! t^{-1/4} \;\! e^{-E_1 t} \;\! \|u_0\|_{w}
$
for every $t \geq 1$. This is a consequence of the following result.
\begin{Proposition}\label{Prop.straight}
There exists a positive constant~$C$ such that for every $t \geq 1$,
$$
  C^{-1} \, t^{-1/4} \;\! e^{-E_1 t}
  \ \leq \
  \big\|e^{-t H_0}\big\|_{\sii_w(\Omega_0) \to \sii(\Omega_0)}
  \ \leq \
  C \, t^{-1/4} \;\! e^{-E_1 t}
  \,.
$$
Moreover,
for every bounded set $B \subset \Omega_0$ and $x \in \Omega_0$
there is a constant $C_{B,x}$ such that for $t\geq 1$,
\begin{displaymath}
  C_{B,x}^{-1} \, t^{-1/2} \, e^{-E_1 t}
  \ \leq \
  \mathbb{P}_x\bigl(X_t \in B,\tau_{\Omega_0} > t\bigr)
  \ \leq \
  C_{B,x} \, t^{-1/2} \,  e^{-E_1 t}
  \,.
\end{displaymath}
\end{Proposition}
\begin{proof}
The second assertion is a rather immediate consequence of~\eqref{kernel.straight}.
In order to see this, observe that
\begin{eqnarray}\label{e:simplebound}
\lefteqn{
  \mathbb{P}_x\bigl(X_t \in B,\tau_{\Omega_0} > t\bigr)
  = \sum_{n=1}^\infty e^{-E_n t} \,
  \mathcal{J}_n(x_2) \, \int_B p(x_1,x_1',t) \, \mathcal{J}_n(x_2')\,dx
  }
  \nonumber \\
  && = e^{-E_1 t} \, \mathcal{J}_1(x_2) \,
  \int_B p(x_1,x_1',t) \, \mathcal{J}_1(x_2')\,d(x_1',x_2')+ R_B(t,x_1,x_2)
  \,,
\end{eqnarray}
%
%
%\begin{equation}\label{e:simplebound}
%\begin{split}
% \mathbb{P}_x\bigl(X_t \in B,\tau_{\Omega_0} > t\bigr)& = \sum_{n=1}^\infty e^{-E_n t} \,
%  \mathcal{J}_n(x_2) \, \int_B p(x_1,x_1',t) \, \mathcal{J}_n(x_2')\,dx \\
%  &= e^{-E_1 t} \,
%  \mathcal{J}_1(x_2) \, \int_B p(x_1,x_1',t) \, \mathcal{J}_1(x_2')\,d(x_1',x_2')+ R_B(t,x_1,x_2)
%\end{split}
%\end{equation}
%
where $R_B(t,x_1,x_2)$
satisfies $|R_B(t,x_1,x_2)|\leq e(x_1,x_2)e^{-E_2 t}$ ($t \geq 1$)
for some locally bounded function $e:\Omega_0\rightarrow \mathbb{R}_+$.
Thus there exists $t_0=t_{0}(x_1,x_2,B)\geq 1$
such that for every $t \geq t_0$ one has
$$
  |R_B(t,x_1,x_2)|\leq \frac{1}{2} e^{-E_1 t} \,
  \mathcal{J}_1(x_2) \, \int_B p(x_1,x_1',t) \,
  \mathcal{J}_1(x_2')\,dx'.
$$
Therefore from \eqref{e:simplebound} we conclude that for $t\geq t_0$
\begin{equation*}
\begin{split}
\frac{1}{2}e^{-E_1 t} \,
  \mathcal{J}_1(x_2) \, &\int_B p(x_1,x_1',t) \, \mathcal{J}_1(x_2')\,dx'
  \leq \mathbb{P}_x\bigl(X_t \in B,\tau_{\Omega_0} > t\bigr)\\
  &\leq \frac{3}{2}e^{-E_1 t} \,
  \mathcal{J}_1(x_2) \, \int_B p(x_1,x_1',t) \, \mathcal{J}_1(x_2')\,dx',
\end{split}
\end{equation*}
which, using the explicit form of $p$,
gives the assertion for $t \geq t_0$.
Adjusting the constants $C_{B,x}$ allows to extend this to $t \geq 1$.

Let us now consider the first assertion. Using the Schwarz inequality, we get
\begin{align*}
  \big\|e^{-t H_0} u_0\big\|^2
  &\leq \|u_0\|_w^2
  \int_{\Omega_0\times\Omega_0} s_0(x,x',t)^2 \, w(x')^{-1} \, dx \, dx'
  \\
  & = \|u_0\|_w^2 \ e^{-2 E_1 t} \
  \sum_{n=1}^{\infty} e^{-2 (E_n-E_1)t}
  \int_{\Real\times\Real} p(x_1,x_1',t)^2 \, e^{-{x_1'}^2/4} \, dx_1 \, dx_1'
\end{align*}
for every $u_0 \in \sii_w(\Omega_0)$ and $t \geq 0$.
Here the sum can be estimated by a constant independent of $t \geq 1$
and the integral (computable explicitly)
is proportional to~$t^{-1/2}$.
This establishes the upper bound of the proposition.

To get the lower bound, we may restrict
to the class of initial data of the form
$u_0(x)=\varphi(x_1)\;\!\mathcal{J}_1(x_2)$ with $\varphi\in\sii_w(\Real)$
(here~$w$ is considered as a function on~$\Real$).
Then it is easy to see from~\eqref{kernel.straight} that
$$
  \big\|e^{-t H_0}\big\|_{\sii_w(\Omega_0) \to \sii(\Omega_0)}
  \geq
  e^{-E_1 t} \,
  \big\|e^{t \Delta^\Real}\big\|_{\sii_w(\Real) \to \sii(\Real)}
$$
for every $t \geq 0$.
The lower bound with~$t^{-1/4}$ is well known
for the heat semigroup of~$-\Delta^\Real$
(or can be easily established by taking
$\varphi=w^{-\alpha}$ with any $\alpha>1/2$
and evaluating the integrals with the kernel~$p$ explicitly).
\end{proof}
\begin{Remark}
It is clear from the proof that the bounds hold
in less restrictive weighted spaces.
Indeed, it is enough to have a corresponding result
for the one-dimensional heat semigroup~$e^{t \Delta^{\Real}}$.
\end{Remark}
For the following Corollary
we recall the definition of the elementary conditional probability.
If the measurable subset $B$ satisfies $\mathbb{P}_x(B)>0$,
then $\mathbb{P}_x(A \mid B) := \mathbb{P}_x(A \cap B)/\mathbb{P}_x(B)$.
The concept of conditional probabilities allows to focus
on the polynomial decay factors, as the exponential terms cancel each other.
\begin{Corollary}
Let $K=0$.
For every bounded measurable subset $B \in \Omega_0$ and every $x \in \Omega_0$ there exists a constant $c_{B,x}>0$ such that
\begin{displaymath}
c_{B,x}^{-1}\,t^{-\frac{1}{2}} \leq \mathbb{P}_x\bigl(X_t \in B \mid \tau_{\Omega_0}>t \bigr) \leq c_{B,x}\,t^{-\frac{1}{2}}
\end{displaymath}
for every $t \geq 1$.
\end{Corollary}
\begin{proof}
The inequalities follow from Proposition \ref{Prop.straight}
and the fact that for every $x=(x_1,x_2) \in \Omega_0$
by independence of the horizontal and vertical components of $(X_t)$ (in the flat case)
\begin{displaymath}
\lim_{t \rightarrow \infty}e^{E_1 t}\mathbb{P}_x\bigl(\tau_{\Omega}>t\bigr)= \mathcal{J}_1(x_2)\int_{(-a,a)}\mathcal{J}_1(x_2)\,dx_2.
\end{displaymath}
From the definition of the conditional probability,
we see that the exponential cancel
and we remain with the polynomial decay as stated in the assertion.
\end{proof}
\noindent
As a consequence of this result, we get that
conditioned on not hitting the boundary $\Omega_0$
the Brownian particle will escape to infinity.

Proposition~\ref{Prop.straight} establishes the decay rate
for zero curvature as announced in Table~\ref{table}.

%--------------------------------------------%
\subsection{The criticality of the transport}\label{Sec.critical}
%--------------------------------------------%
%
Let us now explain what we mean by the vague statement in Table~\ref{table}
that the transport is `critical' on flat surfaces.

We say that the transport is \emph{critical}
if the spectral threshold of~$H_K$ is not
`stable against local attractive perturbations', \ie,
\begin{equation}\label{critical}
  \forall V \in C_0^\infty(\Omega_0)
  \,, \ V\not=0 \,, \ V \leq 0
  \,, \qquad
  \inf\sigma(H_K + V) < \lambda_K
  \,.
\end{equation}
Then we also say that~$H_K$ is critical.
As a consequence of the spectral mapping theorem,
we get
\begin{equation*}
  \big\|e^{-t(H_K+V)}\big\|_{\sii_f(\Omega_0)\to\sii_f(\Omega_0)}
  = e^{\gamma t} e^{-\lambda_K t}
\end{equation*}
for each time $t \geq 0$,
where $\gamma := \lambda_K-\inf\sigma(H_K + V)$ is positive.
That is, the criticality leads to an exponential slow-down
in the decay of the perturbed semigroup.

Property~\eqref{critical} is well known for~$H_0$ and is equivalent
to the fact that the first component of $(X_t)_{t\geq 0}$
-- a one-dimensional Brownian motion -- is recurrent.
For some results concerning this connection in a more abstract context
we refer to \cite{McGO}.
\begin{Proposition}\label{Prop.critical}
$H_0$ is critical.
\end{Proposition}
\begin{proof}
By the variational characterization of the spectral threshold,
it is enough to construct a test function~$\psi$
from $H_0^1(\Omega_0)$ such that
\begin{equation*}
  Q[\psi] :=
  \|\nabla \psi\|^2 - E_1 \|\psi\|^2 - \big\||V|^{1/2}\psi\big\|^2
  < 0
  \,.
\end{equation*}
For every $n\geq 1$, we define
$\psi_n(x) := \varphi_n(x_1) \;\! \mathcal{J}_1(x_2)$,
with $\varphi_n := w^{-n}$,
where~$w$ is the weight~\eqref{weight}
(considered as a function on~$\Real$).
Due to the normalization of~$\mathcal{J}_1$,
we have
$$
  Q[\psi_n] = \|\dot\varphi_n\|_{\sii(\Real)}^2
  - \|v\;\!\varphi_n\|_{\sii(\Real)}^2
  \,,
$$
where $v(x_1) := \| |V(x_1,\cdot)|^{1/2} \mathcal{J}_1\|_{\sii((-a,a))}^2$.
By hypothesis, $v \in L^1(\Real)$
and the integral $\|v\|_{L^1(\Real)}$ is positive.
Finally, an explicit calculation yields
$\|\dot\varphi_n\|_{\sii(\Real)} \sim n^{-1/4}$.
By the dominated convergence theorem, we therefore have
$$
  Q[\psi_n] \xrightarrow[n\to\infty]{} - \|v\|_{L^1(\Real)}
  \,.
$$
Consequently, taking~$n$ sufficiently large,
we can make $Q[\psi_n]$ negative.
\end{proof}

In Section~\ref{Sec.positive} we shall show that the spectrum of~$H_0$
is unstable against purely geometric deformations
characterized by positive curvature, too.

%--------------------------------------%
\section{Asymptotically flat manifolds}\label{Sec.vanish}
%--------------------------------------%
%
We say that the strip~$\Omega$ is \emph{asymptotically flat}
if its Gauss curvature~$K$ vanishes at infinity, \ie,
\begin{equation}\label{vanish}
  \lim_{|x_1| \to \infty} \esssup_{x_2\in(-a,a)} |K(x)| = 0
  \,.
\end{equation}
In this paper, we are interested in a `locally perturbed traveller'
by usually assuming a stronger hypothesis that~$K$ is compactly supported, \ie,
\begin{equation}\label{Ass.compact}
  \supp(K) \cap \Omega_0
  \quad
  \mbox{is bounded}.
\end{equation}
It follows from~\eqref{Ass.compact} that there exists
a positive~$R$ such that $K(x)=0$ for all $|x_1| > R$.
Then, as a consequence of~\eqref{Jacobi},
\begin{equation}\label{consequence.compact.initial}
  |x_1| > R
  \quad\Longrightarrow\quad
  f(x)=1
  \,.
\end{equation}
Of course, \eqref{vanish}~trivially holds for the strips
satisfying~\eqref{Ass.compact}.
Nevertheless, let us state the following result
under the more general hypothesis~\eqref{vanish}.
\begin{Theorem}\label{Thm.ess}
Assume~\eqref{Ass.basic.alt} and~\eqref{vanish}. Then
$$
  \sigma_\mathrm{ess}(H_K) = [E_1,\infty)
  \,.
$$
\end{Theorem}
\begin{proof}
The fact that the threshold of the essential spectrum
does not descend below the energy~$E_1$
has been proved in~\cite[Thm.~1]{Krej1}
by means of a Neumann bracketing argument.
Let us therefore only show that $[E_1,\infty)$
belongs to the essential spectrum of~$H_K$.

Our proof is based on the Weyl criterion adapted
to quadratic forms in~\cite{DDI}
and applied to quantum waveguides in~\cite{KKriz}.
By this general characterization of essential spectrum
and since the set $[E_1,\infty)$ has no isolated points,
it is enough to find for every $\lambda \in [E_1,\infty)$
a sequence
$
  \{\psi_n\}_{n=1}^\infty \subseteq \Dom(h_K)
$
such that
\begin{itemize}
\item[(i)]
$\forall n\in\Nat\setminus\{0\}$, \quad $\|\psi_n\|_{f}=1$,
\item[(ii)]
$\big\|(H_K-\lambda)\psi_n\big\|_{[\Dom(h_K)]^*}
\xrightarrow[n\to\infty]{} 0$.
\end{itemize}
Here~$\|\cdot\|_{[\Dom(h_K)]^*}$ denotes the norm in
the dual space $[\Dom(h_K)]^*$ of~$\Dom(h_K)$
Let $n\in\Nat\setminus\{0\}$.
Given $k\in\Real$, we set $\lambda=E_1+k^2$.

Since~$\Omega$ is asymptotically flat,
a good candidate for the sequence are
plane waves in the $x_1$-direction modulated
by the ground-state eigenfunction~$\mathcal{J}_1$ in the $x_2$-direction
and `localized at infinity':
$$
  \psi_n(x) := \varphi_n(x_1) \, \mathcal{J}_1(x_2) \, e^{ikx_1}
  \,.
$$
Here $\varphi_n(x_1):^{-1/2} \varphi(x_1/n-n)$ with~$\varphi$
being a non-zero $C^\infty$-smooth function
with compact support in the interval $(-1,1)$.
Note that $\supp\varphi_n \subset (n^2-n,n^2+n)$.
We further assume that~$\varphi$
is normalized to~$1$ in $L^2(\Real)$,
so that the norm of~$\varphi_n$ is~$1$ as well.

Clearly, $\psi_n \in H_0^1(\Omega_0)=\Dom(h_K)$.
To satisfy~(i), one can redefine~$\psi_n$ by
dividing it by its norm $\|\psi_n\|_f$.
However, since
$$
  \|\psi_n\|_f^2
  \geq 1 - \frac{\|K\|_\infty \;\! a^2}{1-\|K\|_\infty \;\! a^2}
  > 0
$$
due to Lemma~\ref{Lem.Taylor} and
the normalizations of~$\varphi$ and~$\mathcal{J}_1$,
it is enough to verify the condition~(ii) directly
for our unnormalized functions~$\psi_n$.

By the definition of the dual norm, we have
\begin{equation}\label{dual}
  \big\|(H_K-\lambda)\psi_n\big\|_{[\Dom(h_K)]^*}
  = \sup_{\phi \in H_0^1(\Omega_0) \setminus\{0\}}
  \frac{\big| h_K(\phi,\psi_n)-\lambda \;\! (\phi,\psi_n)_f \big|}
  {\|\phi\|_{\Dom(h_K)}}
  \,.
\end{equation}
An explicit computation using integrations parts yields
\begin{multline*}
  h_K(\phi,\psi_n)-\lambda \;\! (\phi,\psi_n)_{f}
  = \big(\phi,
  [-\ddot\varphi_n-2ik\dot\varphi_n] \;\! \mathcal{J}_1 \;\! e^{ikx_1}
  \big)
  \\
  + \big(\partial_1\phi,[f^{-1}-1]\partial_1\psi_n\big)
  - k^2 \, \big(\phi,[f-1]\psi_n\big)
  - \big(\phi,[\partial_2 f]\partial_2\psi_n\big)
  \,.
\end{multline*}
Using the Schwarz inequality, we estimate the individual terms
on the right hand side of the identity as follows
\begin{align*}
  \big|
  \big(\phi,
  [-\ddot\varphi_n-2ik\dot\varphi_n] \;\! \mathcal{J}_1 \;\! e^{ikx_1}
  \big)
  \big|
  &\leq \|\phi\|_{\Dom(h_K)} \,
  \sqrt{\|\ddot\varphi_n\|_{L^2(\Real)}^2+4k^2\|\dot\varphi_n\|_{L^2(\Real)}^2}
  \, \|f^{1/2}\|_\infty
  \,,
  \\
  \big| \big(\partial_1\phi,[f^{-1}-1]\partial_1\psi\big) \big|
  &\leq \|\phi\|_{\Dom(h_K)} \,
  \|\dot\varphi_n\|_{\sii(\Real)} \,
  \esssup_{\supp\varphi_n} \big( f^{1/2} \, |f^{-1}-1| \big)
  \,,
  \\
  \big| \big(\phi,[f-1]\psi_n\big) \big|
  &\leq \|\phi\|_{\Dom(h_K)} \,
  \|\varphi_n\|_{\sii(\Real)} \,
  \esssup_{\supp\varphi_n} \big( f^{-1/2} \, |f-1| \big)
  \,,
  \\
  \big| \big(\phi,[\partial_2 f]\partial_2\psi_n\big) \big|
  &\leq \|\phi\|_{\Dom(h_K)} \,
  \|\varphi_n\|_{\sii(\Real)} \, E_1 \,
  \esssup_{\supp\varphi_n} \big( f^{-1/2} \, |\partial_2 f| \big)
  \,.
\end{align*}

Hence, the dual norm~\eqref{dual} can be bounded from above
by a constant multiplied by a sum of terms containing
either $\|\dot\varphi_n\|_{\sii(\Real)}$, $\|\ddot\varphi_n\|_{\sii(\Real)}$
or the suprema involving~$f$ over the support of~$\varphi_n$.
By hypothesis~\eqref{vanish},
the suprema tend to zero as $n\to\infty$ due to Lemma~\ref{Lem.Taylor}
and~\eqref{f2.Taylor}.
The remaining terms tend to zero as $n\to\infty$ because
$$
  \|\dot\varphi_n\|_{L^2(\Real)}
  = n^{-1} \, \|\dot\varphi\|_{L^2(\Real)}
  \,, \qquad
  \|\ddot\varphi_n\|_{L^2(\Real)}
  = n^{-2} \, \|\ddot\varphi\|_{L^2(\Real)}
  \,.
$$
\end{proof}

Theorem~\eqref{Thm.ess} implies that we always have $\lambda_K \leq E_1$
for asymptotically flat strips.
Therefore, as a consequence of~\eqref{SMT},
\begin{equation*}
  \big\|e^{-t H_K}\big\|_{\sii_f(\Omega_0)\to\sii_f(\Omega_0)}
  \geq e^{-E_1 t}
\end{equation*}
for each time $t \geq 0$.

%------------------------------------%
\section{Positively curved manifolds}\label{Sec.positive}
%------------------------------------%
%
We say that a manifold is \emph{positively curved}
if~$K$ is non-zero and non-negative
(in the sense of a measurable function on the manifold).
In this section we give a meaning to the vague statement
of Table~\ref{table} that
the `positive curvature is bad for transport'.
It is based on the following result,
which we adopt from~\cite{Krej1}.
\begin{Theorem}\label{Thm.disc}
Assume~\eqref{Ass.basic.alt} and $K \in L^1(\Omega_0)$.
We have
$$
  (\mathcal{J}_1,K\mathcal{J}_1)_f > 0
  %\int_{\Omega_0} K(x) \, \mathcal{J}_1(x_2)^2 \, dx > 0
  \qquad\Longrightarrow\qquad
  \inf\sigma(H_K) < E_1
  \,.
$$
\end{Theorem}
\begin{Remark}\label{Rem.burden}
Recall that~$\mathcal{J}_1$ is the first transverse eigenfunction
introduced in~\eqref{spectrum.straight}.
Here, not to burden the notation,
we denote by the same symbol~$\mathcal{J}_1$
the function $x\mapsto\mathcal{J}_1(x_2)$ on~$\Omega_0$.
\end{Remark}
\begin{proof}
The proof of the theorem is very similar
to that of Proposition~\ref{Prop.critical}.
By the variational characterization of the spectral threshold of~$H_K$,
it is enough to construct a test function~$\psi$
from $H_0^1(\Omega_0)$ such that
\begin{equation*}
  Q_K[\psi] :=
  h_K[\psi] - E_1 \|\psi\|_f^2
  < 0
  \,.
\end{equation*}
Using the same sequence of functions
$\psi_n(x) = \varphi_n(x_1)\mathcal{J}_1(x_2)$
as in the proof of Proposition~\ref{Prop.critical},
we arrive at
\begin{equation}\label{2integrals}
  Q_K[\psi_n] = (\partial_1\psi_n,f^{-1}\partial_1\psi_n)
  - \frac{1}{2} \, (\psi_n,K\psi_n)_f
  \,.
\end{equation}
Here the first (positive) integral on the right hand side
vanishes as $n\to\infty$ because
$$
  (\partial_1\psi_n,f^{-1}\partial_1\psi_n)
  \leq \frac{1-\|K\|_\infty \;\! a^2}{1-2\;\!\|K\|_\infty \;\! a^2}
  \ \|\dot\varphi_n\|_{\sii(\Real)}^2
  \,,
$$
due to Lemma~\ref{Lem.Taylor} and
the normalization of~$\mathcal{J}_1$,
and $\|\dot\varphi_n\|_{\sii(\Real)} \sim n^{-1/4}$.
Using, at the same time, the dominated convergence theorem
in the second integral on the right hand side of~\eqref{2integrals},
we finally get
$$
  Q_K[\psi_n] \xrightarrow[n\to\infty]{}
  - \frac{1}{2} \, (\mathcal{J}_1,K\mathcal{J}_1)_f
  \,.
$$
Since the limit is negative by hypothesis,
we can make $Q[\psi_n]$ negative
by taking~$n$ sufficiently large.
\end{proof}
\begin{Remark}
The integrability of~$K$ is just a technical assumption
in Theorem~\ref{Thm.disc}. It is only important to give
a meaning to the integral $(\mathcal{J}_1,K\mathcal{J}_1)_f$,
the value~$+\infty$ being admissible in principle.
For instance, it is enough to assume that~$K$ is non-trivial
and non-negative on~$\Omega_0$ for the present proof to work.
\end{Remark}

Combining Theorem~\ref{Thm.disc} with Theorem~\ref{Thm.ess},
we get that~$H_K$ possesses at least one discrete eigenvalue
below the essential spectrum under the hypotheses.
In view of the criticality notion introduced in Section~\ref{Sec.critical},
the result of Theorem~\ref{Thm.disc} can be also interpreted
in the sense that~$H_0$ is not stable against geometric
perturbations characterized by the presence of positive curvature.

In any case, regardless of whether the spectral threshold of~$H_K$
represents an eigenvalue or the bottom of the essential spectrum,
Theorem~\ref{Thm.disc} implies that the gap $\gamma:=E_1-\lambda_K$
is always positive for positively curved strips. If $K$ vanishes at
infinity, then the bottom of the spectrum has to be an isolated eigenvalue.
Therefore, as a consequence of~\eqref{SMT} and \cite{Sim}, we conclude with
\begin{Corollary}\label{c:cor-pos-curv}
Assume~\eqref{Ass.basic.alt}, $K \in L^1(\Omega_0)$
and $(\mathcal{J}_1,K\mathcal{J}_1)_f > 0$.
Then
\begin{equation*}
  \big\|e^{-t H_K}\big\|_{\sii_f(\Omega_0)\to\sii_f(\Omega_0)}
  = e^{\gamma t} \, e^{-E_1 t}
\end{equation*}
for each time $t \geq 0$, where~$\gamma$ is positive.
Moreover, if additionally \eqref{vanish} is satisfied then there exists
a unique non-negative normalized
$\phi_0 \in L^2_f(\Omega_0)$ such that
for every bounded measurable set $B \subset \Omega_0$ and every $x \in \Omega_0$
\begin{displaymath}
\lim_{t \rightarrow \infty}e^{-(\gamma-E_1) t} \,
\mathbb{P}_x \bigl(X_t \in B, \tau_{\Omega_0} > t\bigr)
=  \phi_0(x)\int_{B} \phi_{0}(y)f(y)\,dy
  \,.
\end{displaymath}
\end{Corollary}

That is, the presence of positive curvature clearly slows down the decay
of the heat semigroup, even without the need to work with
the weighted space $\sii_{wf}(\Omega_0)$. A Brownian traveller should avoid
`mountains' satisfying $(\mathcal{J}_1,K\mathcal{J}_1)_f > 0$,
if he/she wants to make sure that
he/she is able to reach is goal early
and wants to avoid spending too much time in a given bounded region.

The following Corollary
(again a rather direct consequence of Theorems~\ref{Thm.disc} and~\ref{Thm.ess})
shows that in contrast to the flat case the Brownian traveller
-- conditioned on not-hitting the boundary $\partial \Omega_0$ --
might not have been able to have left a bounded region forever.
\begin{Corollary}
Assume that~\eqref{vanish}
and the conditions of Theorem \ref{Thm.disc} are satisfied.
Then for almost every $x \in \Omega_0$ we have
\begin{displaymath}
\lim_{t \rightarrow \infty}
\mathbb{P}_x\bigl(X_t \in \cdot \mid \tau_{\Omega_0}>t\bigr)
= \frac{\phi_0(y)f(y)\,dy}{\int_{\Omega_0}\phi_0(y)f(y)\,dy},
\end{displaymath}
where the convergence is with respect to the total variation distance.
\end{Corollary}
\begin{proof}
By definition of the total variation distance we have to prove that
\begin{displaymath}
\lim_{t \rightarrow \infty}\sup_{B\subset {\Omega_0}}\biggl|\mathbb{P}_x\bigl(X_t \in B \mid \tau_{\Omega_0}>t\bigr) - \frac{\int_B\phi_0(y)f(y)\,dy}{\int_{\Omega_0}\phi_0(y)f(y)\,dy}\biggr|=0.
\end{displaymath}
Observe that we do not assume that the sets $B$
are bounded and that the assertions of Corollary \ref{c:cor-pos-curv}
do not suffice to prove the desired assertion.

According to general spectral theory we know,
that the eigenfunction $\phi_0 \in L_f^2(\Omega_0)$ does not change sign
and that the eigenspace is one-dimensional.
In the first step we show that $\phi_0$ actually also
belongs to $L_f^1(\Omega_0)$,
with the notation $L_f^p(\Omega_0):=L^p(\Omega_0,f(x)\,dx)$.
This will allow us to interpret the ground state as a probability distribution.
Of course, many results concerning the decay properties are known,
but we have not been able to find a reference covering our setting.
Observe first that due to the probabilistic interpretation
the semigroup $(e^{-tH_K})_{t \geq 0}$ in $L^2_f(\Omega_0)$
gives rise to a consistent strongly continuous semigroups
$(T^p_t)_{t \geq 0}$ in $L_f^p(\Omega_0)$ for $1 \leq p < \infty$.
Moreover, due to the Gaussian bound from Proposition \ref{p:fundprop},
these semigroups are analytic with angle $\pi/2$.
Let the generators be denoted by $H^p_K$.
Due to the consistence of the semigroups,
by taking Laplace transforms of the semigroups,
we conclude that the resolvents
$F^p(z):=(H^p_K-z)^{-1}$
($z \in \rho(H^p_K):=\mathbb{C}\setminus \sigma(H_K^p))$
are as well consistent in the sense that
for every $z \in \rho(H^p_K)\cap \rho(H^q_K)$
\begin{displaymath}
F^p(z) \restriction L_f^p(\Omega_0)\cap L_f^q(\Omega_0)
= F^q(z) \restriction L_f^p(\Omega_0)\cap L_f^q(\Omega_0).
\end{displaymath}
Since according to Theorem 5 in \cite{Dav}
we have $\sigma(H_K^p)=\sigma(H_K^2)(=\sigma(H_K))$
for every $1 \leq p < \infty$ and since $\lambda_K$
is an isolated eigenvalue for $H_K^2$,
we conclude by Corollary 1.4 in \cite{HV}
that $\lambda_K$ is an isolated point of $\sigma(H_K^1)$
and that the eigenvector $\phi_0$ of $H_K^2$
is also an eigenvector of $H_K^1$, \ie, in particular,
$\phi_0 \in L_f^1(\Omega_0)$.

Observe now that $(X_t)_{\tau_{\Omega_0} \geq t \geq 0}$ is $\lambda_K$-recurrent in the sense of \cite{TT} and we also conclude that the measure $\pi(dx)=\phi_0(y)f(y)\,dy$ is finite and due to reversibility with respect to the measure $f(x)\,dx$ satisfies
\begin{displaymath}
\mathbb{P}_{\pi}\bigl(X_t \in A,\tau_{\Omega_0}> t\bigr)=e^{-\lambda_Kt}\int_A\phi_0(x)f(x)\,dx
\end{displaymath}
for every measurable set $A \subset \Omega_0$.
As $\phi \in L_f^2(\Omega_0)$ we conclude that $(X_t)$
is $\lambda_K$-positive recurrent in the sense of \cite{TT}
(product-critical in the sense of \cite{Pin}).
Applying Theorem 7 in \cite{TT},
we are thus able to conclude the assertion of the Corollary.
More precisely, formula (5.9) in \cite{TT}
shows that for almost all $x \in \Omega_0$
\begin{equation}\label{e:TT}
\lim_{t \rightarrow \infty}\sup_{B \subset \Omega_0}
\left|e^{\lambda_Kt}\mathbb{P}_x\bigl(X_t \in B,\tau_{\Omega_0}>t\bigr)
-\phi_0(x)\int_B\phi_0(y)f(y)\,dy\right| = 0.
\end{equation}
Therefore we have for almost all $x \in \Omega_0$
\begin{equation*}
\begin{split}
\sup_{B \subset \Omega_0}&\biggl|\mathbb{P}_x\bigl(X_t \in B\mid
\tau_{\Omega_0}>t\bigr)-\frac{\int_{B}\phi_0(z)f(z)\,dz}
{\int_{\Omega_0}\phi_0(y)f(y)\,dy}\biggr|
\\
\leq\ & \big(e^{\lambda_K t}\mathbb{P}_x \bigl(\tau_{\Omega_0}>t)\big)^{-1}
\sup_{B \subset \Omega_0}
\left|e^{\lambda_K t}\mathbb{P}_x\bigl(X_t \in B)
-\phi_0(x)\int_B\phi_0(y)f(y)\,dy\right|
\\
&+\sup_{B \subset \Omega_0}
\left|\frac{\phi_0(x)\int_B\phi_0(y)f(y)\,dy}
{e^{\lambda_K t}\mathbb{P}_x\bigl(\tau_{\Omega_0}>t)}
-\frac{\int_B\phi_0(y)f(y)\,dy}{\int_{\Omega_0}\phi_0(y)f(y)\,dy}
\right|
\\
=\ & \big(e^{\lambda_K t}\mathbb{P}_x\bigl(\tau_{\Omega_0}>t)\big)^{-1}
\sup_{B \subset \Omega_0}
\left|e^{\lambda_K t}\mathbb{P}_x\bigl(X_t \in B)
-\phi_0(x)\int_B\phi_0(y)f(y)\,dy\right|
\\
&+\biggl(\sup_{B \subset \Omega_0}\int_B\phi_0(z)\,dz\biggr)
\left|\frac{\phi_0(x)}{e^{\lambda_K t}\mathbb{P}_x\bigl(\tau_{\Omega_0}>t)}
-\frac{1}{\int_{\Omega_0}\phi_0(y)f(y)\,dy}\right|.
\end{split}
\end{equation*}
Two applications of \eqref{e:TT} complete the proof.
\end{proof}
%

%------------------------------------%
\section{Negatively curved manifolds}\label{Sec.negative}
%------------------------------------%
%
In analogy with positively curved manifolds,
we say that a manifold is \emph{negatively curved}
if~$K$ is non-zero and non-positive.
In this section, on the contrary,
we show that the presence of negative curvature
improves the decay of the heat semigroup,
supporting in this way the vague statement
of Table~\ref{table} that
the `negative curvature is good for transport'.
First, however, we have to explain why the negative sign of curvature
is much more delicate for the study of large time properties of~\eqref{heat}.

Recall that the positivity of the curvature~$K$
pushes the spectrum below~$E_1$ (\cf~Theorem~\ref{Thm.disc}).
The objective of this subsection is to show that
the effect of negative curvature is rather opposite:
it `has the tendency' to push the spectrum above~$E_1$.
This effect is more subtle because $[E_1,\infty)$
belongs to the spectrum of~$H_K$,
irrespectively of the sign of the curvature,
as long as the curvature vanishes at infinity
(\cf~Theorem~\ref{Thm.ess}).

The way how to understand this `repulsive tendency'
is to replace the Poin\-ca\-r\'e-type inequality requirement
$H_K-E_1 \geq \const > 0$ (which is false for the asymptotically
flat manifolds) by a weaker, \textit{Hardy-type inequality}:
\begin{equation}\label{Hardy}
  H_K - E_1
  \ \geq \
  \rho > 0
  \,.
\end{equation}
Here $\rho: \Omega_0 \to (0,\infty)$ is assumed to be merely
a \emph{positive function}
(necessarily vanishing at the infinity of~$\Omega_0$
for the asymptotically flat manifolds).

By~Theorem~\ref{Thm.disc}, \eqref{Hardy}~is false for
positively curved manifolds.
It is also violated for flat manifolds because
of the criticality result of Proposition~\ref{Prop.critical}.
In this subsection, we show that~\eqref{Hardy} typically holds
for negatively curved manifolds.

%--------------------------------------------------------------%
\subsection{Hardy-type inequality and the large time behaviour}
\label{ss:hardylargetime}
%--------------------------------------------------------------%
%
For completeness we first sketch an abstract elementary argument from~\cite{KW},
relating the Hardy inequality to low energy properties of the Hamiltonian
and the large time behaviour of the semigroup.
If the semigroup is associated to a stochastic process,
then the validity of a Hardy-type inequality is related to the concept
of $R$-transience of the stochastic process.

Assume that there exists a positive function~$\rho$,
with a locally bounded inverse~$\rho^{-1}$,
such that the inequality~\eqref{Hardy}
holds true for the self-adjoint non-negative operator $L_K:=H_K-E_1$.
Then according to Theorem 8.31 in \cite{Weid}
we conclude that for all $\lambda < 0$ and every $h \in L^2_f(\Omega_0)$ we have
\begin{equation}\label{e:Hardyinvers}
\bigl(h,(L_K-\lambda)^{-1}h \bigr)_{f}
%=\bigl(h,(H_K-E_1-\lambda)^{-1}h \bigr)_{f}
\leq \bigl(h,(M_{\rho}-\lambda)^{-1}h\bigr)_{f}
\,,
\end{equation}
where $M_{\rho}$ denotes the maximal multiplication operator
acting via multiplication with the function $\rho$.
If $h$ satisfies
%
%\begin{displaymath}
$
  (h,\rho^{-1}h)_f < \infty
$,
%\end{displaymath}
%
then \eqref{e:Hardyinvers} implies
\begin{equation}\label{e:beforetrans}
\forall \lambda < 0 \,, \qquad
\int_{(E_1,\infty)}(\nu-\lambda)^{-1}\,
d\big\|E_{\nu}^{L_K}h\big\|^2_{f}
\leq \bigl(h,\rho^{-1}h\bigr)_f < \infty
\,,
\end{equation}
where $(E_{\nu}^{L_K})_{\nu}$ denotes
the spectral resolution of~$L_K$.
Using monotone convergence, we get for all $h$ with
$(h,\rho^{-1}h)_f < \infty$
(in particular for all continuous $h$ with compact support in~$\Omega_0$)
\begin{equation}\label{e:e1-transience}
\int_0^{\infty} \big(h,e^{-t (H_K-E_1)}h\big)_{f}\,dt < \infty.
\end{equation}
Observe that \eqref{e:e1-transience}
-- which in the probabilistic literature such as \cite{TT}, \cite{T74a}
and \cite{T74b} might be called $E_1$-transience --
does not hold in the case of positively curved and flat manifolds.

Property \eqref{e:beforetrans} is related to the low energy behaviour
of the spectral measure $E^{L_K}(\cdot)$
in the sense that it implies that for all $r \in (0,1]$
and $-1\leq \lambda < 0$
\begin{equation}\label{spectralmeasure}
\big\|E^{L_K}((0,r))h\big\|_{f}^2
= \int_0^r\,d\|E_{\nu}^{L_K}h\|^2_{f}
\leq \int_{0}^r\frac{r-\lambda}{\nu-\lambda} \, d\|E_{\nu}^{L_K}h\|^2_{f}.
\end{equation}
where we used that $\frac{r-\lambda}{\nu-\lambda} \geq 1$ for $\nu \in (0,r)$
and negative $\lambda$.
Sending $\lambda$ to $0$ and using \eqref{e:beforetrans},
we conclude that there is $C>0$ such that
for $h$ with $(h,\rho^{-1}h)_f \leq 1$ and $r \in (0,1)$
\begin{displaymath}
\big\|E^{L_K}((0,r))h\big\|_{f}^2 \leq C\,r.
\end{displaymath}
This insight can easily be translated into an assertion concerning
the large time behaviour.
\begin{Proposition}\label{p:simpleconsequHardy}
Assume that $H_K-E_1$ satisfies the Hardy-type inequality~\eqref{Hardy}
with a positive function~$\rho$
satisfying $\rho^{-1} \in L_\mathrm{loc}^\infty(\Omega_0)$.
Then
\begin{displaymath}
\sup_{(h,\rho^{-1}h)_{f}<1}
\bigl\|e^{-t(H_K-E_1)}h\bigr\|_{f}^2
\leq  \frac{1}{t} \, \bigl( 1/2 + 2e^{-2}\bigr)
\,.
\end{displaymath}
\end{Proposition}
\begin{proof}
For the proof we again set $L_K:=H_K-E_1$ and denote by $\mu_h$
the spectral measure corresponding to $L_K$ and $h$.
Via the spectral theorem, integration by parts and \eqref{spectralmeasure},
we obtain
\begin{equation*}
\begin{split}
\bigl\| e^{-tL_K} h \bigr\|^2_{f}&=\int_0^{\infty}e^{-2 \nu t}\,d\mu_h(\nu) = 2 t \int_0^{1} e^{-2 \nu t } \,\mu_h(\nu)\,d\nu + 2t\int_1^{\infty}e^{- 2 \nu t}\,\mu_h(\nu)\,d\nu\\
&\le  2 t \int_{0}^1 e^{- 2 \nu t}\nu\,d\mu_h + 2 t \int_{1}^{\infty} e^{-2 \nu t}\,\mu_h(\nu)\,d\nu\\
&\le  2 t \int_{0}^{\infty}e^{-2 \nu t} \nu \,d\nu + t \int_{1}^{\infty} e^{-2 t\nu}\,d\nu \leq \frac{\Gamma(1)}{2t}+ 2 t e^{-2 t} \\
&=\frac{1}{t}\bigl( 1/2 + 2 t^2 e^{-2t}\bigr).
\end{split}
\end{equation*}
Observing that $\max_{t > 0}\bigl(2 t^2 e^{-2t}\bigr) = 2e^{-2}$
yields the desired assertion.
\end{proof}
Observe again that,
under weak conditions on the Hardy weight $\rho$,
Proposition~\ref{p:simpleconsequHardy} already gives
an accelerated decay rate
when compared with the one in the flat case given
in Proposition~\ref{Prop.straight}.

\subsection{The Hardy inequality for negatively curved manifolds}\label{ss:Hardynegcurv}
In this subsection, we show that~\eqref{Hardy} typically holds
for negatively curved manifolds.

One way how to establish~\eqref{Hardy} is to
generalize the method of~\cite{K3}.
It works as follows:
\begin{enumerate}
\item
\emph{Transverse ground-state estimate.}
Recalling the structure of our operator~\eqref{LB},
we clearly have
\begin{equation}\label{crucial.bound}
  H_K - E_1
  \ \geq \
  - f^{-1} \partial_1 f^{-1} \partial_1 + \mu_K
\end{equation}
in the form sense on $\sii_f(\Omega_0)$,
where $x \mapsto \mu_K(x_1)$ denotes the lowest eigenvalue
of the one-dimensional shifted `transverse' operator
$- f^{-1} \partial_2 f \partial_2 - E_1$
%$$
%  - f^{-1} \partial_2 f \partial_2
%  \qquad\mbox{on}\qquad
%  \sii_f\big((-a,a),f(x_1,x_2)\,dx_2\big)
%  \,,
%$$
on the Hilbert space $\sii\big((-a,a),f(x_1,x_2)\,dx_2\big)$,
subject to Dirichlet boundary conditions,
with~$x_1$ being considered as a parameter
in the one-dimensional eigenvalue problem.
More specifically, we have
\begin{equation}\label{lambda}
  \mu_K(x_1)
  = \inf_{\varphi\in H_0^1((-a,a))\setminus\{0\}} \,
  \frac{\int_{-a}^a |\dot{\varphi}(x_2)|^2 \, f(x_1,x_2) \, dx_2}
  {\int_{-a}^a |\varphi(x_2)|^2 \, f(x_1,x_2) \, dx_2}
  \ - \ E_1
  \,.
\end{equation}
With an abuse of notation, we denote by the same symbol~$\mu_K$
both the function on~$\Real$ and its natural extension
$x \mapsto \mu_K(x_1)$ to~$\Omega_0$.
\item
\emph{Longitudinal Hardy-type estimate.}
Now we regard the right hand side of~\eqref{crucial.bound}
as a one-dimensional Schr\"odinger-type operator
on the Hilbert space $\sii\big((-a,a),f(x_1,x_2)\,dx_1\big)$,
with~$x_2$ being considered as a parameter
and~$\mu_K$ playing the role of potential.
We assume that each of the $x_2$-dependent family
of operators satisfies a Hardy-type inequality, so that
\begin{equation}\label{crucial.bound.longitudinal}
  - f^{-1} \partial_1 f^{-1} \partial_1 + \mu_K
  \ \geq \ \rho_K > 0
\end{equation}
in the form sense on $\sii_f(\Omega_0)$,
with some positive function $\rho_K: \Omega_0 \to (0,\infty)$.
Then~\eqref{Hardy} holds as a consequence of~\eqref{crucial.bound.longitudinal}
and~\eqref{crucial.bound}.
\end{enumerate}

In this way, we have reduced the problem to ensuring
the existence of one-dimensional Hardy-type
inequalities~\eqref{crucial.bound.longitudinal}.
However, the criticality of one-dimensional
Schr\"odinger operators is well studied, at least if $f=1$.
We present two sufficient conditions which guarantee
the validity of~\eqref{crucial.bound.longitudinal}
and confirm thus that~\eqref{Hardy} typically holds
for negatively curved manifolds.

\subsubsection{Positivity of the ground-state estimates}
Since the kinetic part of the Schr\"odinger-type operator
on the left hand side of~\eqref{crucial.bound.longitudinal}
is a non-negative operator, we get a trivial estimate
\begin{equation}\label{Hardy.trivial}
  - f^{-1} \partial_1 f^{-1} \partial_1 + \mu_K
  \ \geq \
  \mu_K
\end{equation}
in the form sense on $\sii_f(\Omega_0)$.
As a consequence of~\eqref{crucial.bound}, $H_K-E_1 \geq \mu_K$.

This represents a \emph{local} Hardy-type inequality
provided that~$\mu_K$ is non-zero and non-negative.
By `local' we mean that the function~$\mu_K$ is compactly supported
for manifolds with compactly supported curvature~$K$,
which is a typical hypothesis of the present paper.
Hence it does not fit to the initial definition~\eqref{Hardy},
which can be called \emph{global} Hardy-type inequality.
However, it is known that local Hardy-type inequalities
imply global ones.

\begin{Theorem}[Hardy inequality for non-negative $\mu_K$]\label{Thm.Hardy}
Assume~\eqref{Ass.basic.alt}.
If~$\mu_K$ is non-zero and non-negative
in some bounded open subinterval $J \subset \Real$,
then there exists a positive constant~$c_K$,
depending on~$a$ and properties of~$K$,
such that
\begin{equation}\label{Hardy.mu}
  - f^{-1} \partial_1 f^{-1} \partial_1 + \mu_K
  \ \geq \
  \frac{c_K}{1+\delta^2}
\end{equation}
in the form sense on~$\sii_f(\Omega_0)$.
Here~$\delta(x):=|x_1-x_1^0|$,
with~$x_1^0$ being the mid-point of~$J$.
As a consequence of~\eqref{crucial.bound},
the Hardy-type inequality~\eqref{Hardy} holds.
\end{Theorem}
\begin{proof}
The proof follows by a modification of the proof of \cite[Thm.~3.1]{K3}
(\cf~also \cite[Thm.~6.7]{K6-with-erratum}).
For the clarity of the exposition,
we divide it into several steps.

\smallskip\noindent
\emph{1.~A consequence of the classical Hardy inequality.}
The main ingredient in the proof is the following
Hardy-type inequality for a Schr\"odinger operator
in the strip~$\Omega_0$
with a characteristic-function potential:
\begin{equation}\label{Hardy.classical}
  \|(1+\delta^2)^{-1/2}\psi\|^2
  \leq 16 \, \|\partial_1\psi\|^2
  + (2+64/|J|^2) \, \|\chi_J\psi\|^2
\end{equation}
for every $\psi \in H^1(\Omega_0)$.
Here~$J$ is any bounded open subinterval of~$\Real$
and~$\chi_J$ denotes the characteristic function
of the set $J \times (-a,a) \subset \Omega_0$.
This inequality can be established quite easily (\cf~\cite[Sec.~3.3]{EKK})
by means of Fubini's theorem
and the classical one-dimensional Hardy inequality
$
  \int_0^b s^{-2} |\varphi(s)|^2 ds
  \leq 4 \int_0^b |\dot\varphi(s)|^2 ds
$
valid for any $\varphi\in H^1((0,b))$, $b>0$,
satisfying $\varphi(0)=0$.

Using Lemma~\ref{Lem.Taylor}, \eqref{Hardy.classical}
can be cast into the form
\begin{equation}\label{Hardy.classical.bis}
  \big\|f^{-1}\partial_1\psi\big\|_f^2
  + \big\|\mu_K^{1/2} \psi\big\|_f^2
  \geq c \, \|(1+\delta^2)^{-1/2}\psi\|_f^2
  - C \, \|\chi_J\psi\|_f^2
  \,,
\end{equation}
where the constants are given by
\begin{align*}
  c := \frac{1-\|K\|_\infty \;\! a^2}{16}
  \,, \qquad
  C := \left(\frac{1}{8}+\frac{4}{|J|^2}\right)
  \left[
  1-\left(\frac{\|K\|_\infty \;\! a^2}{1-\|K\|_\infty \;\! a^2}\right)^2
  \right]^{-1}
  .
\end{align*}

\smallskip\noindent
\emph{2.~A Poincar\'e-type inequality in a bounded strip.}
For every $\psi \in H^1(\Omega_0)$, we have
\begin{align}\label{Poincare}
  \big\|f^{-1}\partial_1\psi\big\|_f^2
  + \big\|\mu_K^{1/2} \psi\big\|_f^2
  &\geq
  \big\|\chi_Jf^{-1}\partial_1\psi\big\|_{f}^2
  + \big\|\chi_J\mu_K^{1/2} \psi\big\|_{f}^2
  \nonumber \\
  &\geq \lambda_J \, \big\|\chi_J\psi\big\|_{f}^2
  \,,
\end{align}
where~$\lambda_J$ denoted the lowest eigenvalue
of the operator $- f^{-1} \partial_1 f^{-1} \partial_1 + \mu_K$
on $\sii_f(J\times(-a,a))$, subject to Neumann-type
(\ie~no in the form setting)
boundary conditions at $(\partial J) \times (-a,a)$.
We claim that~$\lambda_J$ can be bounded from below
by a positive constant which depends exclusively
on properties of~$\mu_K$.
Indeed, assume $\lambda_J=0$.
By the variational characterization of~$\lambda_J$,
it follows that
$$
  \big\|\chi_Jf^{-1}\partial_1\psi_J\big\|_{f}^2 = 0
  \qquad\mbox{and}\qquad
  \big\|\chi_J\mu_K^{1/2} \psi_J\big\|_{f}^2 = 0
  \,,
$$
where $\psi_J \in H^1(J\times(-a,a))$ is an eigenfunction
corresponding to~$\lambda_J$.
Recalling Lemma~\ref{Lem.Taylor}, we conclude that
$\|\mu_K\|_{L^1(J)} = 0$, which contradicts the hypothesis
that~$\mu_K$ is non-trivial on~$J$.

\smallskip\noindent
\emph{3.~Some interpolation.}
Combining~\eqref{Hardy.classical.bis} with~\eqref{Poincare},
we eventually arrive at
$$
  \big\|f^{-1}\partial_1\psi\big\|_f^2
  + \big\|\mu_K^{1/2} \psi\big\|_f^2
  \geq c \, \epsilon \, \|(1+\delta^2)^{-1/2}\psi\|_f^2
  + \left[(1-\epsilon)\lambda_J - C \, \epsilon\right] \, \|\chi_J\psi\|_f^2
$$
for every $\psi \in H^1(\Omega_0)$ and any $\epsilon \in (0,1)$.
Choosing~$\epsilon$ in such a way that the term with
the square brackets vanishes, we get the Hardy-type
inequality of the theorem with
$
  c_K := c \lambda_J/(\lambda_J+C)
$.
\end{proof}

\subsubsection{On the positivity of the ground-state eigenvalue}%
\label{Sec.positivity}
Since the fundamental hypothesis of Theorem~\ref{Thm.Hardy}
is the non-negativity of~$\mu_K$, let us comment on
its relation to the non-positivity of~$K$.

We claim that the function~$\mu_K$ is typically positive
for negatively curved manifolds.
Indeed, for any fix $x_1\in\Real$,
let us make the change of test function
$\phi:=\sqrt{f(x_1,\cdot)}\,\varphi$ in~\eqref{lambda}.
Integrating by parts and using~\eqref{Jacobi},
one easily arrives at
\begin{equation}\label{lambda.bis}
  \mu_K(x_1)
  = \inf_{\phi\in H_0^1((-a,a))\setminus\{0\}}
  \frac{\int_{-a}^a
  \left(
  |\dot{\phi}(x_2)|^2 - E_1\;\!|\phi(x_2)|^2 + V(x)\;\!|\phi(x_2)|^2
  \right) dx_2}
  {\int_{-a}^a |\phi(x_2)|^2 \, dx_2}
\end{equation}
with
\begin{equation}\label{potential}
  V := -\frac{1}{2} \, K
  + \frac{1}{4} \left(\frac{\partial_2f}{f}\right)^2
  \,.
\end{equation}
By the Poincar\'e inequality for the Dirichlet Laplacian in $\sii((-a,a))$,
we therefore get
\begin{equation}\label{mu.estimate}
  \mu_K(x_1) \geq \essinf_{x_2\in(-a,a)} V(x_1,x_2)
  \,.
\end{equation}

Let us assume for a moment that~$K$ is continuous.
Then, for every $x_1\in\Real$ fixed,
it follows from~\eqref{Jacobi} that
$$
  \lim_{a \to 0} V(x) = -\frac{1}{2} \, K(x_1,0)
  \,.
$$
Hence, if $K(x_1,x_2)$ is negative for every $x_2 \in (-a,a)$
and~$x_1$ from a compact interval~$J$,
there exists a positive half-width~$a$ such that $\mu_K(x_1)$ is positive
for every $x_1 \in J$.
For merely bounded curvature~$K$, we replace the pointwise
non-positivity requirement on the curve~$\Gamma$ by the hypothesis
that the function
\begin{equation}\label{negativity.replacement}
  k(x_1) :=
  \lim_{a \to 0} \essinf_{x_2\in(-a,a)} K(x_1,x_2)
\end{equation}
is non-zero and non-positive.

It is less obvious how to get uniform lower bounds,
\ie~to ensure that, for a given~$a$,
$\mu_K(x_1)$ is non-negative for almost every $x_1\in\Real$.
An example of manifolds for which the uniform non-negativity
is possible to check is given by strips on ruled surfaces
studied in~\cite{K3}.
\begin{Example}[Ruled strips]\label{Ex.ruled}
Let~$\Gamma$ be a straight line in~$\Real^3$;
without loss of generality, we may assume that
$\Gamma(x_1)=(x_1,0,0)$.
Given a $C^1$-smooth function $\theta:\Real\to\Real$,
let us define
$
  \mathcal{L}(x) :=
  \big(x_1,x_2\cos\theta(x_1),x_2\sin\theta(x_1)\big)
$.
The image~\eqref{image} is a ruled surface,
composed of segments of length~$2a$ translated
and rotated along~$\Gamma$.
It is straightforward to check
that the corresponding metric~$G$ admits the block form~\eqref{metric}
with the explicit formulae
$$
  f(x) = \sqrt{1 + \dot{\theta}(x_1)^2 \,x_2^2}
  \,, \qquad
  K(x) =
  - \frac{\dot{\theta}(x_1)^2}{f(x)^{4}}
  \,.
$$
The \emph{ad hoc} defined mapping~$\mathcal{L}$
represents an explicit parametrization of~$\Omega$
via the exponential map~\eqref{exp}.
The hypothesis~\eqref{Ass.basic.alt} is satisfied for every~$a$
provided that we assume that~$\dot\theta$ is bounded.
The strip~$\Omega$ is asymptotically flat
if $\dot\theta(x_1)$ tends to zero as $|x_1|\to\infty$.
Finally, an explicit calculation yields
\begin{equation}\label{potential.ruled}
  V(x) =
  \frac{\dot{\theta}(x_1)^2
  \big[2-\dot{\theta}(x_1)^2 \, x_2^2\big]}
  {4\,f(x)^4}
  \,.
\end{equation}
It follows that~$V$ is non-zero and non-negative provided
that~$\dot\theta$ is non-zero and the half-width~$a$
is so small that $\|\dot{\theta}\|_\infty \;\! a < \sqrt{2}$.
Consequently, under the same assumptions about~$a$ and~$\dot\theta$,
the quantity~$\mu_K$ is non-zero and non-negative, too.
We refer to~\cite{K3} for more geometric and spectral
properties of the ruled strips.
\end{Example}

\subsubsection{Thin strips}
The second sufficient condition which guarantees the validity
of~\eqref{crucial.bound.longitudinal} is based on the ideas
of the previous subsection.

\begin{Theorem}[Hardy inequality for thin strips]\label{Thm.Hardy.thin}
Assume~\eqref{Ass.basic.alt} and~\eqref{Ass.compact}.
Let the function~$k$ defined in~\eqref{negativity.replacement}
be non-zero and non-positive.
Then there exists a positive number~$a_0$,
depending on properties of~$K$,
such that
\begin{equation}\label{Hardy.mu.bis}
  - f^{-1} \partial_1 f^{-1} \partial_1 + \mu_K
  \ \geq \
  \frac{\tilde{c}_K}{1+x_1^2}
\end{equation}
holds in the form sense on~$\sii_f(\Omega_0)$ for all $a \leq a_0$
with some constant~$\tilde{c}_K$ depending on properties of~$K$.
As a consequence of~\eqref{crucial.bound},
the Hardy-type inequality~\eqref{Hardy} holds for all $a \leq a_0$.
\end{Theorem}
\begin{proof}
In view of~\eqref{mu.estimate}, Lemma~\ref{Lem.Taylor}, \eqref{f2.Taylor}
and~\eqref{consequence.compact.initial},
it is easy to show that
$$
  \mu_K(x_1) \geq
  - \frac{1}{2} \;\! k - C(\|K\|_\infty\;\!a^2) \, \chi_{[-R,R]}(x_1)
  \,,
$$
for almost every $x_1\in\Real$, where
$$
  C(\xi) := \frac{1}{4} \xi^2
  \left(1+\frac{\xi^2}{1-\xi^2}\right)^2
  \left(1-\frac{\xi^2}{1-\xi^2}\right)^{-2}
  \,.
$$
Hence,
$
  \mu_K \to \mu_K^0 := -\frac{1}{2} k
$
as $a \to 0$.
For every $\psi \in H^1(\Omega_0)$, we write
$$
  \big\|f^{-1}\partial_1\psi\big\|_f^2
  + \big(\psi,\mu_K \psi\big)_f
  = \big\|f^{-1}\partial_1\psi\big\|_f^2
  + \big(\psi,\mu_K^0 \psi\big)_f
  + \big(\psi,[\mu_K-\mu_K^0] \psi\big)_f
  \,.
$$
Applying Theorem~\ref{Thm.Hardy} to the first
two terms on the right hand side of this identity,
we get
\begin{multline*}
  \big\|f^{-1}\partial_1\psi\big\|_f^2
  + \big(\psi,\mu_K \psi\big)_f
  \\
  \geq \int_{\Omega_0}
  \left[\frac{c_K}{1+x_1^2}-C(\|K\|_\infty\;\!a^2)\,\chi_{[-R,R]}(x_1)\right]
  |\psi(x)|^2 \, f(x) \, dx
  \,.
\end{multline*}
It is important to notice that~$c_K$ can be bounded from below
by a positive constant independent of~$a$
(\cf~proof of Theorem~\ref{Thm.Hardy}).
On the other hand, $C(\|K\|_\infty\;\!a^2)$ tends to zero as $a \to 0$.
Then the result follows by estimating the characteristic function
by $(1+x_1^2)^{-1}$ multiplied by a constant smaller than~$c_K$ for all
sufficiently small~$a$.
\end{proof}
\begin{Remark}\label{Rem.uniform}
The positive function~$\rho$ on the right hand side
of~\eqref{Hardy} can in principle
vanish on the boundary of~$\partial\Omega_0$.
The objective of this remark is to show that,
if~\eqref{Hardy} holds,
with an arbitrary positive function~$\rho$,
there is also an inequality of the type~\eqref{crucial.bound.longitudinal}
with the right hand side which is independent of
the `transverse' variable~$x_2$.
This can be seen as follows.
Assume~\eqref{crucial.bound.longitudinal} and~\eqref{Ass.compact}.
For any $\psi \in H_0^1(\Omega_0)$ and $\epsilon \in (0,1)$,
we write
\begin{align*}
  h_K[\psi] - E_1 \;\! \|\psi\|_f^2
  & = \epsilon \;\! \big(h_K[\psi] - E_1 \;\! \|\psi\|_f^2 \big)
  + (1-\epsilon) \big(h_K[\psi] - E_1 \;\! \|\psi\|_f^2 \big)
  \\
  & \geq \epsilon \;\! \big(\|\partial_2\psi\|_f^2 - E_1 \;\! \|\psi\|_f^2\big)
  + (1-\epsilon) \;\! \|\rho^{1/2}\psi\|_f^2
  \\
  & = \epsilon \;\! \big(
  \|\partial_2\phi\|^2 - E_1 \;\! \|\phi\|^2 + (\phi,V\phi)
  \big)
  + (1-\epsilon) \;\! \|\rho^{1/2}\phi\|^2
  \\
  & \geq \epsilon \;\!
  \big(\psi,[V+\lambda_\epsilon]\psi\big)
  \,.
\end{align*}
Here the last equality follows by the change of test function
$\phi:=\sqrt{f}\,\psi$, as in Section~\ref{Sec.positivity},
and $x \mapsto \lambda_\epsilon(x_1)$ denotes
the lowest eigenvalue of the one-dimensional operator
$-\partial_2^2 - E_1 + (1-\epsilon) \epsilon^{-1} \rho(x_1,\cdot)$
on $\sii((-a,a))$, subject to Dirichlet boundary conditions,
with~$x_1$ considered as a parameter.
More specifically, we have
$$
  \lambda_\epsilon(x_1)
  := \inf_{\varphi\in H_0^1((-a,a))\setminus\{0\}}
  \frac{\int_{-a}^a
  \left(
  |\dot{\varphi}(x_2)|^2 - E_1\;\!|\varphi(x_2)|^2
  + \frac{1-\epsilon}{\epsilon} \rho(x_1,x_2)\;\!|\varphi(x_2)|^2
  \right) dx_2}
  {\int_{-a}^a |\phi(x_2)|^2 \, dx_2}
  \,.
$$
Since~$K$ has bounded support, it is also true for~$V$,
\cf~\eqref{consequence.compact.initial}.
On the other hand, since~$\rho(x)$ is positive for almost every $x \in \Omega_0$,
$\lambda_\epsilon(x_1)$ is positive for almost every $x_1 \in \Real$.
Furthermore, $\lambda_\eps(x_1)$ tends to infinity as $\epsilon \to 0$
for almost every $x_1 \in \Real$.
Consequently, for sufficiently small~$\epsilon$,
$V+\lambda_\epsilon$ can be bounded from
below by a positive function which depends on~$x_1$ only.
\end{Remark}

Finally, let us emphasize that Theorem~\ref{Thm.Hardy.thin}
covers a very general class of manifolds, not necessarily
negatively curved. It is only important that the manifold
is `negatively curved in the vicinity of the reference curve'~$\Gamma$,
\cf~\eqref{negativity.replacement}.

%-------------------------------%
\subsection{The fine decay rate}
%-------------------------------%

As in the flat case in Proposition \ref{Prop.straight},
we again restrict the class of initial data to the weighted spaces of
$\sii_{wf}(\Omega_0) \subset \sii_f(\Omega_0)$
and consider the following (polynomial) \emph{decay rate} quantity:
\begin{multline}\label{rate}
  \Gamma_K
  := \sup \Big\{ \Gamma\in\Real \ \Big| \,\
  \exists C_\Gamma > 0, \, \forall t \geq 0, \
  \\
  \big\|e^{-(H_K-E_1)t}\big\|_{
  \sii_{wf}(\Omega_0)
  \to
  \sii_f(\Omega_0)
  }
  \leq C_\Gamma \, (1+t)^{-\Gamma}
  \Big\}
  .
\end{multline}

Sections~\ref{ss:hardylargetime} and~\ref{ss:Hardynegcurv}
already imply that
%under the assumptions~\eqref{Ass.basic.alt} and~\eqref{Ass.compact}
the heat semigroup decays faster than in the flat case
provided that the Hardy-type inequality~\eqref{Hardy} holds.
It follows from Proposition~\ref{Prop.straight}
that we have $\Gamma_0 = 1/4$ (\ie~for $K=0$),
whereas Proposition~\ref{p:simpleconsequHardy}
gives $\Gamma_K \geq 1/2$ if~\eqref{Hardy} is satisfied.
%if~\eqref{Ass.basic.alt} and~\eqref{Ass.compact} are satisfied.

The abstract arguments leading to Proposition~\ref{p:simpleconsequHardy}
do not give the precise additional polynomial decay rate.
The objective of the following subsections
is to show that~$\Gamma_K$ is in fact three times bigger
whenever the curvature~$K$ is non-zero and non-positive.

In probabilistic terms we are interested in the precise decay exponent
\begin{multline}\label{rate-prob}
\gamma_K(x,B):= \sup \Big\{ \gamma\in\Real \ \Big| \,\
  \exists \tilde{C}_\gamma > 0, \, \forall t \geq 0, \
  \\
  \mathbb{P}_x\bigl(X_t \in B,\tau_{\Omega_0}>t\bigr)
  \leq C_\gamma \, (1+t)^{-\Gamma}
  \Big\}
  .
\end{multline}
where $x \in \Omega_0, B\subset\subset \Omega_0$. Again we find that
the non-zero and non-positive situation differs from the straight
manifold by a factor $3$. This is the meaning of the last item in Table~\ref{table}.

In view of~\eqref{rate},
it is more convenient to study the shifted heat equation
\begin{equation}\label{heat.shifted}
  \left\{
  \begin{aligned}
    \partial_t u + H_K u - E_1 u &= 0
    &&\mbox{in} \quad \Omega_0 \times (0,\infty)
    \,,
    \\
    u &= u_0
    &&\mbox{on} \quad \Omega_0 \times \{0\}
    \,,
  \end{aligned}
  \right.
\end{equation}
in the functional setting on $\sii_f(\Omega_0)$
as explained in Section~\ref{Sec.dynamics}.
Indeed, \eqref{heat.shifted}~is obtained from~\eqref{heat}
by the replacement
$
  u(t) \mapsto e^{-E_1 t} \, u(t)
$,
with help of the Fermi coordinates.

%----------------------------------------------%
\subsection{The self-similarity transformation}
%----------------------------------------------%
%
Our method to study the asymptotic behaviour of
the heat equation~\eqref{heat} in the presence of curvature
is to adapt the technique of self-similar solutions
used in the case of the heat equation in the whole Euclidean space
by Escobedo and Kavian~\cite{Escobedo-Kavian_1987}
to the present problem.
We closely follow the approach of the recent papers~\cite{KZ1,KZ2},
where the technique is applied to twisted waveguides
in three and two dimensions, respectively.

We perform the self-similarity transformation
in the first (longitudinal) space variable only,
while keeping the other (transverse) space variable unchanged.
More precisely, given $s\in(0,\infty)$,
let us consider the change of function defined by
$$
  (U_s\psi)(y) := e^{s/4} \psi(e^{s/2}y_1,y_2)
  \,.
$$
It defines a unitary transformation
from $\sii_f(\Omega_0)$ to $\sii_{f_s}(\Omega_0)$,
where
\begin{equation}\label{fs}
  f_s(y) := f(e^{s/2}y_1,y_2)
  \,.
\end{equation}

Now we associate to every solution
$
  u \in \sii_\mathrm{loc}\big((0,\infty),dt;\sii_f(\Omega_0)\big)
$
of~\eqref{heat} a `self-similar' solution
$
  \tilde{u}(s):=U_s[u(e^s-1)]
$
in a new $s$-time weighted space
$
  \sii_\mathrm{loc}\big((0,\infty),e^s ds;\sii_{f_s}(\Omega_0)\big)
$.
We have
\begin{equation}\label{SST}
  \tilde{u}(y_1,y_2,s)
  = e^{s/4} u(e^{s/2}y_1,y_2,e^s-1)
\end{equation}
and the inverse change of variables is given by
$$
  u(x_1,x_2,t)
  = (t+1)^{-1/4} \, \tilde{u}\big((t+1)^{-1/2}x_1,x_2,\log(t+1)\big)
  \,.
$$
Note that the original space-time variables $(x,t)$
are related to the `self-similar' space-time variables $(y,s)$
via the relations
\begin{equation}\label{space-times}
\begin{aligned}
  (x_1,x_2,t) &= (e^{s/2} y_1,y_2,e^s-1)
  \,, \\
  (y_1,y_2,s) &= \big((t+1)^{-1/2}x,y_2,\log(t+1)\big)
  \,.
\end{aligned}
\end{equation}
Hereafter we consistently use the notation for respective variables
to distinguish the two space-times.

It is easy to check that this change of variables
transfers the weak formulation of~\eqref{heat.shifted}
to the evolution problem
\begin{equation}\label{heat.weak.similar}
  \big\langle
  \tilde{v}, \tilde{u}'(s)
  -\mbox{$\frac{1}{2}$} \, y_1 \partial_{1}\tilde{u}(s)
  \big\rangle_{\!f_s}
  + \tilde{a}_{s}\big(\tilde{v},\tilde{u}(s)\big) = 0
  \,,
\end{equation}
for each $\tilde{v} \in H_0^1(\Omega_0)$ and a.e.~$s\in[0,\infty)$,
with $\tilde{u}(0) = \tilde{u}_0 := U_0 u_0 = u_0$.
Here $\langle\cdot,\cdot\rangle_{f_s}$
stands for the pairing of $H_0^1(\Omega_0,G_s)$
and its dual $[H_0^{1}(\Omega_0,G_s)]^*$,
where~$G_s$ is the metric of the form~\eqref{metric}
with~$f$ being replaced by~$f_s$,
and~$\tilde{a}_{s}(\cdot,\cdot)$ denotes the sesquilinear form
associated with
\begin{equation}\label{form.a.tilde}
\begin{aligned}
  \tilde{a}_{s}[\tilde{u}] &:=
  \|f_s^{-1}\partial_1\tilde{u}\|_{f_s}^2
  + e^s \;\! \|\partial_2\tilde{u}\|_{f_s}^2
  - e^s \;\! E_1 \;\! \|\tilde{u}\|_{f_s}^2
  - \frac{1}{4} \, \|\tilde{u}\|_{f_s}^2
  \,,
  \\
  \tilde{u} \in \Dom(\tilde{a}_{s}) &:= H_0^1(\Omega_0)
  \,.
\end{aligned}
\end{equation}
More specifically, $H_0^1(\Omega_0,G_s)$ denotes the completion
of $C_0^\infty(\Omega_0)$ with respect to the norm
$
  \|\cdot\|_{\Dom(h_{K_s})} :=
  (
  h_{K_s}[\cdot] + \|\cdot\|_{f_s}^2
  )^{1/2}
$,
where~$h_{K_s}$ is defined as~\eqref{form}
with~$f$ being replaced by~$f_s$.

\begin{Remark}
Note that~\eqref{heat.weak.similar}
is a parabolic equation with $s$-time-dependent coefficients.
The same occurs and has been previously analysed
for the heat equation in the twisted waveguides \cite{KZ1,KZ2},
for the heat equation in the plane with magnetic field~\cite{K7}
and also for a convection-diffusion equation in the whole space
but with a variable diffusion coefficient
\cite{Escobedo-Zuazua_1991,Duro-Zuazua_1999}.
A careful analysis of the behaviour of the underlying elliptic operators
as~$s$ tends to infinity leads to a sharp decay rate for its solutions.
An important difference of the present problem with respect to
the previous works is that also the Hilbert space becomes
time-dependent after the self-similarity transformation,
which makes the analysis substantially more difficult.
\end{Remark}
%

%--------------------------------------------------%
\subsection{The setting in weighted Sobolev spaces}
%--------------------------------------------------%
%
Since~$U_s$ acts as a unitary transformation,
it preserves the space norm of solutions
of~\eqref{heat.shifted} and~\eqref{heat.weak.similar}, \ie,
\begin{equation}\label{preserve}
  \|u(t)\|_{f} = \|\tilde{u}(s)\|_{f_s}
  \,.
\end{equation}
This means that we can analyse the asymptotic time behaviour
of the former by studying the latter.

However, the natural space to study the evolution~\eqref{heat.weak.similar}
is not $\sii_{f_s}(\Omega_0)$
but rather the weighted space $\sii_{wf_s}(\Omega_0)$
with the Gaussian weight~\eqref{weight}.
Following the approach of~\cite{KZ1}
based on a theorem of J.~L.~Lions~\cite[Thm.~X.9]{Brezis_FR}
about weak solutions of parabolic equations
with time-dependent coefficients,
it can be shown that~\eqref{heat.weak.similar}
is well posed in the scale of Hilbert spaces
\begin{equation}\label{scale}
  H_0^{1}(\Omega_0, w \;\! G_s)
  \subset \sii_{wf_s}(\Omega_0) \subset
  \big[H_0^{1}(\Omega_0, w \;\! G_s)\big]^*
  \,.
\end{equation}
Here $H_0^{1}(\Omega_0, w \;\! G_s)$ denotes the completion
of $C_0^\infty(\Omega_0)$ with respect to the norm
$
  (
  h_{K_s}[w^{1/2}\cdot\;\!]
  + \|\cdot\|_{w f_s}^2
  )^{1/2}
$.

More precisely, choosing $\tilde{v}:= w v$
for the test function in~\eqref{heat.weak.similar},
where $v \in C_0^\infty(\Omega_0)$ is arbitrary,
we can formally cast~\eqref{heat.weak.similar}
into the form
\begin{equation}\label{heat.weak.weighted}
  \big\langle v, \tilde{u}'(s) \big\rangle_w
  + a_s\big(v,\tilde{u}(s)\big) = 0
  \,.
\end{equation}
Here $\langle\cdot,\cdot\rangle_w$ denotes the pairing of
$H_0^{1}(\Omega_0, w \;\! G_s)$ and $[H_0^{1}(\Omega_0, w \;\! G_s)]^*$
and $a_s(\cdot,\cdot)$ denotes the sesquilinear form associated with
\begin{equation}\label{form.a}
\begin{aligned}
  a_s[\tilde{u}] \ := \ &
  \|f_s^{-1}\partial_1\tilde{u}\|_{w f_s}^2
  + e^s \;\! \|\partial_2\tilde{u}\|_{w f_s}^2
  - e^s \;\! E_1 \;\! \|\tilde{u}\|_{w f_s}^2
  - \frac{1}{4} \, \|\tilde{u}\|_{w f_s}^2
  \\
  & + \frac{1}{2} \left(y_1 \tilde{u},
  [f_s^{-2}-1] \;\! \partial_1\tilde{u}\right)_{w f_s}
  \,,
  \\
  \Dom(a_s) \ := \ & H_0^{1}(\Omega_0, w)
  \,,
\end{aligned}
\end{equation}
with $H_0^{1}(\Omega_0, w)$ denoting the closure of
$C_0^\infty(\Omega_0)$ with respect to the weighted Sobolev norm
$
  (
  \|\nabla\cdot\|_w^2
  + \|\cdot\|_{w}^2
  )^{1/2}
$.
Note the appearance of the extra term with respect to~\eqref{form.a.tilde}
(it makes the form~$a_s$ non-symmetric if the Hilbert space
$\sii_{wf}(\Omega_0)$ is considered to be complex).

By `formally' we mean that the formulae are meaningless in general,
because the solution~$\tilde{u}(s)$ and its derivative~$\tilde{u}'(s)$
may not belong to
$H_0^{1}(\Omega_0, w \;\! G_s)$
and $[H_0^{1}(\Omega_0, w \;\! G_s)]^*$,
respectively.
The justification of~\eqref{heat.weak.similar}
being well posed in the scale~\eqref{scale}
consists basically in checking the boundedness
and a coercivity of the form~$a_s$ defined on $\Dom(a_s)$
and in noticing that the time-dependent spaces
$\sii_{wf_s}(\Omega_0)$ and $H_0^{1}(\Omega_0, w \;\! G_s)$
coincide with $\sii_{w}(\Omega_0)$ and $H_0^{1}(\Omega_0,w)$,
respectively, as vector spaces.
It is straightforward by using~\eqref{Ass.basic.alt}
and Lemma~\ref{Lem.Taylor}.

%-------------------------------------------%
\subsection{Reduction to a spectral problem}
%-------------------------------------------%
%
Choosing $v := \tilde{u}(s)$ in~\eqref{heat.weak.weighted},
we arrive at the identity
\begin{equation}\label{formal}
  \frac{1}{2} \frac{d}{ds} \|\tilde{u}(s)\|_{w f_s}^2
  = - \hat{l}_s[\tilde{u}(s)]
  \,,
\end{equation}
where $\hat{l}_s[\tilde{u}] := \Re\{a_s[\tilde{u}]\}$,
$\tilde{u} \in \Dom(\hat{l}_s) := \Dom(a_s) = H_0^1(\Omega_0,w)$
(independent of~$s$ as a vector space).
It remains to analyse the coercivity of~$\hat{l}_s$.

More precisely, as usual for energy estimates,
we replace the right hand side of~\eqref{formal}
by the spectral bound, valid for each fixed $s \in [0,\infty)$,
\begin{equation}\label{spectral.reduction}
  \forall \tilde{u} \in \Dom(\hat{l}_s) \;\!, \qquad
  \hat{l}_s[\tilde{u}]
  \geq \nu_K(s) \, \|\tilde{u}\|_{w f_s}^2
  \,,
\end{equation}
where~$\nu_K(s)$ denotes the lowest point in the spectrum of
the self-adjoint operator~$\hat{L}_s$
associated on~$\sii_{wf_s}(\Omega_0)$ with~$\hat{l}_s$;
it depends on the curvature~$K$ through the dependence on~$f$.
Then~\eqref{formal} together with~\eqref{spectral.reduction} implies
the exponential bound
\begin{equation}\label{spectral.reduction.integral}
  \forall s \in [0,\infty) \;\!, \qquad
  \|\tilde{u}(s)\|_{w f_s}
  \leq \|\tilde{u}_0\|_{w f_0} \
  e^{-\int_0^s \nu_K(r) \, dr}
  \,.
\end{equation}

Finally, recall that the exponential bound in~$s$
transfers to a polynomial bound in the original time~$t$,
\cf~\eqref{space-times}.
In this way, the problem is reduced to a spectral analysis
of the family of operators $\{\hat{L}_s\}_{s \geq 0}$.

%-------------------------------%
\subsection{Removing the weight}
%-------------------------------%
%
In order to investigate the operator~$\hat{L}_s$
on $\sii_{w f_s}(\Omega_0)$,
we first map it into a unitarily equivalent operator
$L_s := \mathcal{U} \hat{L}_s \mathcal{U}^{-1}$
on $\sii_{f_s}(\Omega_0)$ via the unitary transform
$$
  \mathcal{U}\;\!\tilde{u} := w^{1/2}\,\tilde{u}
  \,.
$$
By definition, $L_s$~is the self-adjoint operator
associated on $\sii_{f_s}(\Omega_0)$ with the quadratic form
$
  l_s[v] := \hat{l}_s[\mathcal{U}^{-1}v]
$,
$
  v \in \Dom(l_s) := \mathcal{U}\,\Dom(\hat{l}_s)
$.
A straightforward calculation yields
\begin{equation}\label{J0.form}
\begin{aligned}
  l_s[v] \ = \ &
  \|f_s^{-1}\partial_1 v\|_{f_s}^2
  + e^s \;\! \|\partial_2 v\|_{f_s}^2
  - e^s \;\! E_1 \;\! \|v\|_{f_s}^2
  - \frac{1}{4} \, \|v\|_{f_s}^2
  \\
  & - \frac{1}{2} \, (y_1 v,\partial_1 v)_{f_s}
  + \frac{1}{16} \left(
  y_1 v, [2-f_s^{-2}] \;\! y_1 v
  \right)_{f_s}
  \,.
\end{aligned}
\end{equation}
Here and in the sequel, we assume that~$v$ is real,
which is justified by the positivity preserving property
of the heat equation as explained in Section~\ref{Sec.dynamics}.

For everywhere vanishing curvature, \ie~$K=0$,
we have that~$f$ is identically equal to one.
Consequently, $f_s=1$ for all $s \geq 0$.
Then, integrating by parts in the first term
on the second line of~\eqref{J0.form},
we get that~$l_s$ coincides
with the form~$l_s^0$ on $\sii(\Omega_0)$ defined by
\begin{equation}\label{J0.form.flat}
\begin{aligned}
  l_s^0[v] &:=
  \|\partial_1 v\|^2
  + e^s \;\! \|\partial_2 v\|^2
  - e^s \;\! E_1 \;\! \|v\|^2
  + \frac{1}{16} \, \|y_1v\|^2
  \,,
  \\
  \Dom(l_s^0) &:= H_0^1(\Omega_0) \cap \sii(\Omega_0,y_1^2\,dy)
  \,.
\end{aligned}
\end{equation}

In order to specify the domain of~$l_s$ for any curvature,
we assume~\eqref{Ass.compact}
and consider~$l_s$ as a perturbation of~$l_s^0$.
It follows from~\eqref{consequence.compact.initial} that
\begin{equation}\label{consequence.compact}
  |y_1| > e^{-s/2} R
  \quad\Longrightarrow\quad
  f_s(y)=1
  \,.
\end{equation}
In particular, $f_s(y)=1$ for all $|y_1| > R$.

\begin{Lemma}\label{Lem.domain}
Assume~\eqref{Ass.basic.alt} and~\eqref{Ass.compact}.
Then
$$
  \Dom(l_s) = \Dom(l_s^0)
  = H_0^1(\Omega_0) \cap \sii(\Omega_0,y_1^2\,dy)
  \,.
$$
\end{Lemma}
\begin{proof}
Using some rearrangement and integration by parts,
it is convenient to rewrite~\eqref{J0.form} as follows
\begin{equation}\label{J0.form.bis}
  l_s[v] =
  \|f_s^{-1}\partial_1 v\|_{f_s}^2
  + e^s \;\! \|\partial_2 v\|_{f_s}^2
  - e^s \;\! E_1 \;\! \|v\|_{f_s}^2
  + \frac{1}{16} \, \|y_1v\|_{f_s}^2
  + r_s[v]
  \,,
\end{equation}
where
\begin{equation*}
  r_s[v] :=
  - \frac{1}{4} \, \big(v,[f_s-1]v\big)
  - \frac{1}{2} \, \big(y_1 v,[f_s-1]\partial_1 v\big)
  - \frac{1}{16} \, \big(
  y_1 v, [f_s^{-1}-f_s] \;\! y_1 v
  \big)
  \,.
\end{equation*}
Using Lemma~\ref{Lem.Taylor}, it is easy to see that
for each $s \geq 0$ there exists a positive constant
$C=C(s,\|K\|_\infty,a)$
such that
$$
  C^{-1} \, l_s^0[v] \leq l_s[v]-r_s[v] \leq C \, l_s^0[v]
$$
for every $v \in \Dom(l_s^0)$.
Consequently (see, \eg, \cite[Corol.~4.4.3]{Davies}),
the quadratic form $l_s-r_s$
is closed on the domain~$\Dom(l_s^0)$ given by~\eqref{J0.form.flat}.
It remains to show that~$r_s$ is a relatively bounded perturbation
of~$l_s^0$ with relative bound smaller than one.
It is clear for the first term of~$r_s$ which is in fact
a bounded perturbation in view of Lemma~\ref{Lem.Taylor}.
We employ~\eqref{consequence.compact} to deal with the remaining terms.
For the second term we have,
\begin{align*}
  \big| \big(y_1 v,[f_s-1]\partial_1 v\big) \big|
  &\leq \|f-1\|_\infty \int_{\{|y_1|<e^{-s/2} R\}} |y_1| |v(y)| |\partial_1 v(y)| dy
  \\
  &\leq \|f-1\|_\infty \, e^{-s/2} R \, \|v\| \, \|\partial_1 v\|
  \\
  &\leq \|f-1\|_\infty \, R \,
  \left( \epsilon^{-1} \|v\|^2 + \epsilon \;\! \|\partial_1 v\|^2  \right)
\end{align*}
for every $v \in \Dom(l_s^0)$.
and any $\epsilon\in(0,1)$. Similarly,
\begin{equation*}
  \big| \big(y_1 v, [f_s^{-1}-f_s] \;\! y_1 v\big) \big|
  \leq \|f^{-1}-f\|_\infty \, R^2 \, \|v\|^2
  \,.
\end{equation*}
for every $v \in \Dom(l_s^0)$.
\end{proof}
\begin{Remark}
The proof of the lemma represents a direct way how to show
that the form~\eqref{J0.form} is closed on the domain $\Dom(l_s^0)$.
In view of the unitary equivalence~$\mathcal{U}$,
it also \emph{a posteriori} establishes the closedness of
the form~\eqref{form.a}.
\end{Remark}

As a consequence of Lemma~\ref{Lem.domain},
we get that~$L_s$ (and therefore~$\hat{L}_s$) has compact resolvent
and thus purely discrete spectrum for all $s \geq 0$.
In particular, $\nu_K(s)$~represents the lowest
eigenvalue of~$L_s$.

%---------------------------------------------%
\subsection{The strong-resolvent convergence}
%---------------------------------------------%
%
In order to study the decay rate via~\eqref{spectral.reduction.integral},
we need information about the limit of the eigenvalue~$\nu_K(s)$
as the time~$s$ tends to infinity.
This can be deduced from the asymptotic properties
of the resolvent of~$L_s$ for large~$s$.

In view of~\eqref{Ass.compact}, the function $y \mapsto f_s(y)$
converges to one locally uniformly in
$|y_1|>0$, $y_2\in(-a,a)$, as $s\to\infty$.
Moreover, the scaling of the transverse variable in~\eqref{J0.form}
corresponds to considering the operator~$L_0$
in the shrinking strip $\Real \times (-e^{-s/2}a,e^{-s/2}a)$.
This suggests that~$L_s$ will converge, in a suitable sense,
to the one-dimensional harmonic-oscillator operator
\begin{equation}\label{HO}
  h :=
  -\frac{d^2}{dy_1^2} + \frac{1}{16} \, y_1^2
  \qquad \mbox{on} \qquad
  \sii(\Real)
\end{equation}
(\ie\ the Friedrichs extension
of this operator initially defined on $C_0^\infty(\Real)$),
potentially subjected to an extra condition at the origin.
For further purposes, let us note that
the spectrum of~$h$ is known explicitly
(see any book on quantum mechanics, \eg, \cite[Sec.~2.3]{Griffiths})
\begin{equation}\label{HO.spec}
  \sigma(h) = \left\{ \frac{1}{2}
  \left(n+\frac{1}{2}\right)
  \right\}_{n=0}^\infty
  \,.
\end{equation}

We shall see that the difference between the negatively curved
and flat case consists in that the limit operator
for the former is subjected to an extra Dirichlet boundary condition at $y_1=0$.
Thus, simultaneously to~$h$ introduced in~\eqref{HO},
let us consider the self-adjoint operator~$h_D$ in~$\sii(\Real)$
whose quadratic form acts in the same way as that of~$h$
but has a smaller domain
$$
  \Dom(h_D^{1/2}) :=
  \big\{
  \varphi\in\Dom(h^{1/2})\ |\ \varphi(0)=0
  \big\}
  \,.
$$

To make this singular operator limits mentioned above rigorous
($L_s$ and~$h$ act on different spaces),
we decompose the Hilbert space~$\sii(\Omega_0)$
into an orthogonal sum
\begin{equation}\label{Hilbert.decomposition}
  \sii(\Omega_0) = \mathfrak{H}_1 \oplus \mathfrak{H}_1^\bot
  \,,
\end{equation}
where the subspace~$\mathfrak{H}_1$ consists of functions
of the form $\psi_1(y) = \varphi(y_1)\mathcal{J}_1(y')$.
Recall that~$\mathcal{J}_1$ denotes the positive
eigenfunction of $-\Delta_D^{(-a,a)}$ corresponding to~$E_1$,
normalized to~$1$ in $\sii((-a,a))$,
\cf~\eqref{spectrum.straight}.
Given any $\psi \in \sii(\Omega_0)$, we have the decomposition
$\psi = \psi_1 + \psi^\bot$ with $\psi_1\in\mathfrak{H}_1$ as above
and $\psi^\bot \in \mathfrak{H}_1^\bot$.
The mapping $\pi:\varphi\mapsto\psi_1$
is an isomorphism of $\sii(\Real)$ onto~$\mathfrak{H}_1$.
Hence, with an abuse of notations,
we may identify any operator~$h$ on $\sii(\Real)$
with the operator $\pi h \pi^{-1}$
acting on $\mathfrak{H}_1 \subset \sii(\Omega_0)$.

Finally, we mention that the Hilbert spaces $\sii(\Omega_0)$
and $\sii_{f_s}(\Omega_0)$ can be identified as vector sets
because their norms are equivalent.
More specifically, in view of Lemma~\ref{Lem.Taylor} and the definition~\eqref{fs},
we have
\begin{equation}\label{equivalent}
  1 - \frac{\|K\|_\infty \;\! a^2}{1-\|K\|_\infty \;\! a^2}
  \leq \frac{\,\|\psi\|_{f_s}^2}{\|\psi\|^2} \leq
  1 + \frac{\|K\|_\infty \;\! a^2}{1-\|K\|_\infty \;\! a^2}
  \,,
\end{equation}
for every non-zero $\psi \in \sii(\Omega_0)$ and all $s \geq 0$.

In the flat case, \ie~$K=0$,
it is readily seen that the operator~$L_s^0$
associated with the form~\eqref{J0.form.flat}
can be identified with the decomposed operator
\begin{equation}\label{s.decomposition}
  h \otimes 1 + 1 \otimes
  (- e^s \;\! \Delta_{D}^{(-a,a)} - e^s \;\! E_1)
  \qquad\mbox{in}\qquad
  \sii(\Real) \otimes \sii((-a,a))
  \,,
\end{equation}
where~$1$ denotes the identity operators in the appropriate spaces.
Using~\eqref{HO.spec}, it follows that $\nu_0(s) = 1/4$
for all $s \in [0,\infty)$. Consequently,
\begin{equation}\label{untwisted.infinity}
  \nu_0(\infty) := \lim_{s\to\infty} \nu_0(s) = 1/4 \,.
\end{equation}
Moreover, \eqref{s.decomposition}~can be used to show
that~$L_s^0$ converges to $h \oplus 0^\bot$
in the norm-resolvent sense as $s \to \infty$,
where~$0^\bot$ denotes the zero operator on~$\mathfrak{H}_1^\bot$.

It is more difficult (and more interesting) to establish
the asymptotic behaviour of~$\nu_K(s)$ for $K\not=0$.
A fine analysis of its limit leads to the key observation of the paper,
ensuring a gain of~$1/2$ in the decay rate in the negatively curved case.
This can be understood from the following proposition,
which represents the main auxiliary result of the present paper.
\begin{Proposition}\label{Prop.strong}
Assume~\eqref{Ass.basic.alt} and~\eqref{Ass.compact}.
Let the Hardy-type inequality \eqref{Hardy} holds.
Then the operator $L_s$ converges to
$
  h_D \oplus 0^\bot
$
in the strong-resolvent sense as $s \to \infty$, \ie,
\begin{equation*}
  \forall F \in \sii(\Omega_0) \,, \qquad
  \lim_{s \to \infty}
  \left\|
  \big(L_s+i\big)^{-1} F
  - \left[\big(h_D + i \big)^{-1} \oplus 0^\bot\right] F
  \right\|
  = 0
  \,.
\end{equation*}
\end{Proposition}
%
%\begin{Remark}
%In addition to the extended notation for~$\mathcal{J}_1$
%mentioned in Remark~\ref{Rem.burden}, here we denote by
%the same symbol~$\mu_K$ the function $x\mapsto\mu(x_1)$ on~$\Omega_0$,
%with~$\mu_K$ initially introduced in~\eqref{lambda}.
%Recall that~$\mu_K$ is typically non-negative for
%non-positively curved strips,
%the ruled strips of Example~\ref{Ex.ruled} being an example.
%\end{Remark}
%
\begin{proof}
For the clarity of the exposition,
we divide the proof into several steps.
The equivalence of norms~\eqref{equivalent}
and other consequences of Lemma~\ref{Lem.Taylor}
are widely used in the present proof.

\smallskip\noindent
\emph{1.~The resolvent equation for~$L_s$.}
Let $F \in \sii(\Omega_0)$. Then also $F \in \sii_{f_s}(\Omega_0)$
for every $s \geq 0$ due to~\eqref{equivalent}.
Let~$z$ be a sufficiently large positive number
to be specified later.
We set $\psi_s := (L_s+z)^{-1}F$.
In other words, $\psi_s$~satisfies the resolvent equation
\begin{equation}\label{re}
  \forall v \in \Dom(l_s) \,, \qquad
  l_s(v,\psi_s) + z \, (v,\psi_s)_{f_s}
  = (v,F)_{f_s}
  \,.
\end{equation}
In particular, choosing~$\psi_s$ for the test function~$v$ in~\eqref{re},
we have
\begin{equation}\label{resolvent.identity}
  \forall v \in \Dom(l_s) \,, \qquad
  l_s[\psi_s] + z \, \|\psi_s\|_{f_s}^2
  = (\psi_s,F)_{f_s} \leq \|\psi_s\|_{f_s} \|F\|_{f_s}
  \,.
\end{equation}

\smallskip\noindent
\emph{2.~Boundedness of~$\psi_s$.}
Our primary objective is to deduce from~\eqref{resolvent.identity} that
$\{\psi_s\}_{s \geq 0}$ is a bounded family in the space
$
  \mathfrak{D}_0 :=
  H_0^1(\Omega_0) \cap \sii(\Omega_0,y_1^2\,dy)
$
equipped with the intersection topology.

We search a lower bound to the operator $L_s+z$.
Using the convenient form~\eqref{J0.form.bis} for $l_s[\psi_s]$
and proceeding as in the proof of Lemma~\ref{Lem.domain},
we easily check that
\begin{equation}\label{s.psi}
  |r_s[\psi_s]| \leq C \,
  \big(
  \epsilon \;\! \|f_s^{-1}\partial_1\psi_s\|_{f_s}^2
  + \epsilon^{-1}\|\psi_s\|_{f_s}^2
  \big)
  \,.
\end{equation}
with any $\epsilon\in(0,1)$, where~$C$ is a positive constant
depending on $\|K\|_\infty\;\!a^2$ and~$R$.
Hence
\begin{align}\label{s.lower}
  l_s[\psi_s] + z \, \|\psi_s\|_{f_s}^2
  \ \geq \ &
  (1-2\epsilon) \;\! \|f_s^{-1}\partial_1\psi_s\|_{f_s}^2
  + e^s \;\! \|\partial_2 \psi_s\|_{f_s}^2
  - e^s \;\! E_1 \;\! \|\psi_s\|_{f_s}^2
  \nonumber \\
  & + \epsilon \;\! \|f_s^{-1}\partial_1\psi_s\|_{f_s}^2
  + \frac{1}{16} \, \|y_1\psi_s\|_{f_s}^2
  + (z - C\;\!\epsilon^{-1}) \, \|\psi_s\|_{f_s}^2
  \,.
\end{align}
If we choose~$z$ larger than $C\;\!\epsilon^{-1}$,
all the terms on the second line are non-negative.

To get a non-negative lower bound to the first line
on the right hand side of~\eqref{s.lower},
we introduce a new function~$u_s$ by
$\psi_s(y)=e^{s/4} u_s(e^{s/2} y_1,y_2)$
(\cf~the self-similarity transformation~\eqref{SST}).
Making the change of variables $(x_1,x_2)=(e^{s/2} y_1,y_2)$,
recalling the definition~\eqref{lambda}
and using the Hardy-type inequality~\eqref{Hardy},
we obtain
\begin{eqnarray}\label{unself}
\lefteqn{
  (1-2\epsilon) \;\! \|f_s^{-1}\partial_1\psi_s\|_{f_s}^2
  + e^s \;\! \|\partial_2 \psi_s\|_{f_s}^2
  - e^s \;\! E_1 \;\! \|\psi_s\|_{f_s}^2}
  \nonumber \\
  && = e^s \;\! (1-2\epsilon) \;\! \|f^{-1}\partial_1 u_s\|_{f}^2
  + e^s \;\! \|\partial_2 u_s\|_{f}^2
  - e^s \;\! E_1 \;\! \|u_s\|_{f}^2
  \nonumber \\
  && \geq e^s \;\! (1-2\epsilon) \;\! \|\rho^{1/2} u_s\|_{f}^2
  + e^s \;\! 2 \epsilon \, (u_s,\mu_{K}u_s)_{f}
  \,.
\end{eqnarray}
Here $\rho$~is a positive function and,
as pointed out in Remark~\ref{Rem.uniform},
we may assume that it depends on~$x_1$ only.
On the other hand, $\mu_K$~has compact support
due to~\eqref{consequence.compact.initial}.
Hence, we can choose~$\epsilon$ sufficiently small so that
the new Hardy weight
$
  \tilde{\rho}(x_1) := (1-2\epsilon) \rho(x_1) + 2 \epsilon \mu_{K}(x_1)
$
is positive for almost every $x_1 \in \Real$.
Coming back to our coordinates~$y$, we thus conclude from~\eqref{unself}
\begin{equation*}%\label{unself}
  (1-2\epsilon) \;\! \|f_s^{-1}\partial_1\psi_s\|_{f_s}^2
  + e^s \;\! \|\partial_2 \psi_s\|_{f_s}^2
  - e^s \;\! E_1 \;\! \|\psi_s\|_{f_s}^2
  \geq e^s \;\! \|\tilde{\rho}_s^{1/2}\psi_s\|_f^2
  \,,
\end{equation*}
where $\tilde{\rho}_s(y_1):=\rho(e^{s/2}y_1)$.

Using the last inequality in~\eqref{s.lower}
and employing Lemma~\ref{Lem.Taylor},
we eventually arrive at
\begin{equation}\label{s.lower.ultimate}
  l_s[\psi_s] + z \, \|\psi_s\|_{f_s}^2
  \geq
  c \left(
  e^s \;\! \|\tilde{\rho}_s^{1/2}\psi_s\|^2
  + \|\partial_1\psi_s\|^2
  + \|y_1\psi_s\|^2
  + (z - C\;\!\epsilon^{-1}) \, \|\psi_s\|^2
  \right)
  ,
\end{equation}
where~$c$ is a positive constant depending on $\|K\|_\infty\;\!a^2$.
Comparing this inequality with~\eqref{resolvent.identity},
we see that there exists a constant~$z_0$,
depending on~$a$ and properties of~$K$,
such that for all $z \geq z_0$
\begin{equation}\label{s.bounded}
  \|\psi_s\| \leq C \, \|F\| \,,
  \qquad
  \|y_1\psi_s\| \leq C \, \|F\|  \,,
  \qquad
  \|\partial_1\psi_s\| \leq C \, \|F\| \,,
\end{equation}
and
\begin{equation}\label{s.delta.scaling}
  e^s \;\! \|\tilde{\rho}_s^{1/2}\psi_s\|^2
  \leq C \, \|F\|^2 \,,
\end{equation}
with some constant~$C$ depending on~$a$ and properties of~$K$.
Furthermore, directly from~\eqref{s.lower} and~\eqref{resolvent.identity}
with help of~\eqref{s.bounded}, we also get
\begin{equation}\label{s.bounded.transverse}
  \|\partial_2\psi_s\| \leq C \, \|F\| \,,
\end{equation}

The estimate~\eqref{s.lower.ultimate}
also shows that $L_s+z$ is invertible for all $z \geq z_0$.
This, \emph{a posteriori}, justifies the definition of~$\psi_s$
as the unique solution of~\eqref{re}.

From~\eqref{s.bounded} and~\eqref{s.bounded.transverse},
we conclude that
$\{\psi_s\}_{s \geq 0}$ is a bounded family in~$\mathfrak{D}_0$.
Therefore it is precompact in the weak topology of~$\mathfrak{D}_0$.
Let~$\psi_\infty$ be a weak limit point,
\ie, for an increasing sequence of positive numbers $\{s_n\}_{n\in\Nat}$
such that $s_n \to \infty$ as $n \to \infty$,
$\{\psi_{s_n}\}_{n\in\Nat}$
converges weakly to~$\psi_\infty$ in $\mathfrak{D}_0$.
Actually, we may assume that it converges strongly in $\sii(\Omega_0)$
because $\mathfrak{D}_0$ is compactly embedded in $\sii(\Omega_0)$.

\smallskip\noindent
\emph{3.~Transverse mode decomposition of~$\psi_s$.}
Now we employ the Hilbert space decomposition~\eqref{Hilbert.decomposition}
and write
$
  \psi_s(y)=\varphi_s(y_1)\mathcal{J}_1(y_2) + \psi_s^\bot(y)
$,
where $\psi_s^\bot \in \mathfrak{H}_1^\bot$, \ie,
\begin{equation}\label{orthogonality}
  %\forall y_1 \in \Real \,, \qquad
  \big(\mathcal{J}_1,\psi_s^\bot(y_1,\cdot)\big)_{\sii((-a,a))} = 0
\end{equation}
for a.e.\ $y_1 \in \Real$.
It follows from~\eqref{s.bounded}, \eqref{s.bounded.transverse}
 and~\eqref{orthogonality} that
also $\{\psi_s^\bot\}_{s \geq 0}$ is a bounded family in~$\mathfrak{D}_0$
and that $\{\varphi_s\}_{s \geq 0}$ is a bounded family in
$H^1(\Real)\cap\sii(\Real,y_1^2 \, dy_1)$
equipped with the intersection topology.
We denote by~$\psi_\infty^\bot$ and $\varphi_\infty$
the respective limit points.

We come back to~\eqref{resolvent.identity} with~\eqref{s.lower}
and focus on the inequality
\begin{equation}\label{s.transverse.bis}
  e^s \;\! \|\partial_2 \psi_s\|_{f_s}^2
  - e^s \;\! E_1 \;\! \|\psi_s\|_{f_s}^2
  \leq C \, \|F\|^2
\end{equation}
we have already used to get~\eqref{s.bounded.transverse}.
In the same way as we proceeded to get~\eqref{lambda.bis},
we write $\phi_s := \sqrt{f_s}\;\!\psi_s$ and obtain
\begin{equation}\label{s.potential}
  \|\partial_2 \psi_{s}\|_{f_{s}}^2
  - E_1 \;\! \|\psi_{s}\|_{f_{s}}^2
  = \|\partial_2 \phi_{s}\|^2
  - E_1 \;\! \|\phi_{s}\|^2
  + (\phi_s,V_s\;\!\phi_s)
  \,,
\end{equation}
where~$V_s$ is defined in the same way as~\eqref{potential}
but with~$K$ and~$f$ being replaced by~$K_s$ and~$f_s$, respectively.
Using~\eqref{consequence.compact}, it is possible to check that
$\{\phi_{s_n}\}_{n\in\Nat}$ is strongly converging in $\sii(\Omega_0)$;
in fact,
\begin{equation}\label{s.convergence}
  \lim_{n\to\infty} \|\phi_{s_n}-\psi_\infty\| = 0
  \,.
\end{equation}

Using the fact that the scaled potential~$V_s$ in~\eqref{s.potential}
vanishes for all $|y_1| > e^{-s/2} R$ together with the strong convergence
of~$\{\phi_{s_n}\}_{n\in\Nat}$,
it is easy to see that the integral containing the potential
tends to zero as $n\to\infty$,
after passing to the subsequence~$\{s_n\}_{n\in\Nat}$.
Multiplying~\eqref{s.transverse.bis} by~$e^{-s_n}$
and putting the asymptotically vanishing integral
on the right hand side of the inequality, we thus get
\begin{equation}\label{s.translimit}
  \lim_{n\to\infty}
  \left(
  \|\partial_2 \phi_{s_n}\|^2
  - E_1 \;\! \|\phi_{s_n}\|^2
  \right)
  = 0
  \,.
\end{equation}

Using in addition the Hilbert space decomposition~\eqref{Hilbert.decomposition}
of~$\phi_s$, \ie\ $\phi_s(y)=\eta_s(y_1)\mathcal{J}_1(y_2)+\phi_s^\bot(y)$,
we see that the same limit~\eqref{s.translimit}
holds for $\phi_{s_n}^\bot \in \mathfrak{H}_1^\bot$ as well.
In that limit, we use the estimate
$
  \|\partial_2 \phi_{s_n}^\bot\|^2
  \geq E_2 \;\! \|\phi_{s_n}^\bot\|^2
$,
where $E_2=4E_1$ denotes the second eigenvalue of $-\Delta_D^{(-a,a)}$,
and conclude that $\|\phi_{s_n}^\bot\|$ tends to zero as $n\to\infty$.
The latter together with~\eqref{s.convergence} finally yields
\begin{equation}\label{s.translimit.ultimate}
  \lim_{n\to\infty}
  \|\psi_{s_n}^\bot\|
  = 0
  \qquad\mbox{and}\qquad
  \lim_{n\to\infty}
  \|\eta_{s_n}-\varphi_\infty\|_{\sii(\Real)}
  = 0
  \,.
\end{equation}
That is, $\psi_\infty \in \mathfrak{H}_1$.

\smallskip\noindent
\emph{4.~The Dirichlet condition at zero.}
Now we come back to the inequality~\eqref{s.delta.scaling}.
Recall that $\tilde{\rho}_{s}(y_1)=\tilde{\rho}_{s}(e^{s/2}y_1)$
and that~$\tilde{\rho}$ is positive
(although necessarily vanishing at infinity).
Without loss of generality, we may assume that~$\tilde{\rho}$
in~\eqref{s.delta.scaling} belongs to $L^1(\Real)$
(since we can always replace the estimate using a smaller function).
Then $e^{s/2} \tilde{\rho}_{s}$ converges in the sense of distributions
on~$\Real$ to a Dirac delta at $y_1=0$.
We want to use this heuristic consideration to show
that $\varphi_\infty(0)=0$.

To do so, first, we use the Hilbert space
decomposition~\eqref{Hilbert.decomposition} of~$\psi_s$
and notice that the left hand side of~\eqref{s.delta.scaling}
splits into a sum of two non-negative parts,
the mixed term being zero due to~\eqref{orthogonality}.
Second, multiplying the obtained inequality for
the term involving the $\mathfrak{H}_1$-part of~$\psi_s$ by $e^{-s/2}$,
passing to the subsequence~$\{s_n\}_{n\in\Nat}$
and taking the limit $n\to\infty$,
we arrive at
\begin{equation*}
  |\varphi_\infty(0)|^2 \int_\Real \tilde{\rho}(x_1) \, dx_1
  = 0
  \,.
\end{equation*}
The limiting procedure is justified by recalling that
$\{\varphi_{s_n}\}_{n\in\Nat}$ converges weakly in $H^1(\Real)$
and therefore strongly in $H^1(J)$,
which is compactly embedded in $C^{0,\lambda}(J)$
for every $\lambda\in(0,1/2)$,
where~$J$ is any bounded interval of~$\Real$.

Since the integral of~$\tilde{\rho}$ is positive by our hypotheses,
we thus verify that~$\varphi_\infty$ satisfies the extra
Dirichlet condition
\begin{equation*}
  \varphi_\infty(0) = 0
  \,.
\end{equation*}

\smallskip\noindent
\emph{5.~The resolvent equation for~$L_s$ as $s\to\infty$.}
Let us summarize our results. We have obtained that
the solutions~$\psi_{s_n}$ of~\eqref{re} converge
in the weak topology of~$\mathfrak{D}_0$
and in the strong topology of $\sii(\Omega_0)$
to some~$\psi_\infty$.
Moreover, the limiting solution~$\psi_\infty$ belongs to~$\mathfrak{H}_1$,
so that $\psi_\infty(y)=\varphi_\infty(y_1)\mathcal{J}_1(y_2)$
with some
$
  \varphi_\infty \in H^1(\Real)\cap\sii(\Real,y_1^2 \, dy_1)
  = \Dom(h)
$.
Finally, $\varphi_\infty(0)=0$,
so that actually $\varphi_\infty \in \Dom(h_D)$.

Recall that the set $C_0^\infty(\Real\!\setminus\!\{0\})$
is dense in $\Dom(h_D)$.
Let $\varphi \in C_0^\infty(\Real\!\setminus\!\{0\})$ be arbitrary.
We take $v(y) := \varphi(y_1) \mathcal{J}_1(y_2)$
as the test function in~\eqref{re}
and note that~$\varphi$ and $f_s-1$ have disjoint support
for sufficiently large~$s$ due to~\eqref{consequence.compact}.
Sending~$n$ to infinity in~\eqref{re} with~$s$ being replaced by~$s_n$,
we thus easily check that
\begin{equation*}
  (\dot\varphi,\dot\varphi_\infty)_{\sii(\Real)}
  + \frac{1}{16} \, (y_1\varphi,y_1\varphi_\infty)_{\sii(\Real)}
  + z \, (\varphi,\varphi_\infty)_{\sii(\Real)}
  = (\varphi,f)_{\sii(\Real)}
  \,,
\end{equation*}
where $f(y_1) := (\mathcal{J}_1,F(y_1,\cdot))_{\sii((-a,a))}$.
That is, $\varphi_\infty = (h_D+z)^{-1} f$,
for \emph{any} weak limit point of $\{\varphi_s\}_{s \geq 0}$.

In conclusion, we have shown that~$\psi_{s}$
converges strongly to $\psi_\infty$
in $\sii(\Omega_0)$ as $s \to \infty$,
where
$
  \psi_\infty(y)
  := \varphi_\infty(y_1)\mathcal{J}_1(y_2)
  = \big[(h_D+z)^{-1} \oplus 0^\bot \big] F(y)
$.

\smallskip\noindent
\emph{6.~The strong convergence for other values of the spectral parameter.}
Finally, let us argue that we can replace the real number~$z$
by any non-real number. This is actually a consequence of
\cite[Thm.~VIII.1.3]{Kato}, the fact that~$L_s$ is
self-adjoint on $\sii_{f_s}(\Omega_0)$
and the equivalence of this Hilbert space with $\sii(\Omega_0)$,
to which we consider the limit of the strong convergence,
due to~\eqref{equivalent}.
\end{proof}
\begin{Remark}\label{Rem.unself}
The crucial step in the proof is certainly the usage
of the Hardy inequality~\eqref{Hardy}.
Indeed, it enables us, first, to ensure
the non-negativity of the right hand side of~\eqref{s.lower}
and, second, to establish the extra Dirichlet condition at zero.
\end{Remark}
%

%----------------------------------%
\subsection{Spectral consequences}
%----------------------------------%
%
Assume for a moment that Proposition~\ref{Prop.strong} stated
that the operator~$L_s$ converges to $h_D \oplus 0^\bot$
in the \emph{norm-resolvent} sense as $s\to\infty$.
Then we would immediately know that~$\nu_K(s)$ converges
to the first eigenvalue of~$h_D$ as $s \to \infty$.
In view of the symmetry, the first eigenvalue of~$h_D$
coincides with the second eigenvalue of~$h$,
which is~$3/4$ due to~\eqref{HO.spec}.
Hence, under the hypotheses of proposition~\ref{Prop.strong},
we would have that the limit of~$\nu_K(s)$ as $s\to\infty$
is three-times larger than the same
limit in the flat case~\eqref{untwisted.infinity}.

Unfortunately,
the strong-resolvent convergence of Proposition~\ref{Prop.strong}
is not sufficient to guarantee the convergence of spectra.
In general, this is true for eigenvalues of the limiting operator
which are \emph{stable} under the perturbation (\cf~\cite[Sec.~VIII.1]{Kato}).
In our case, however, the spectral convergence can be established
directly using the fact that both~$L_s$ and~$h_D$ are operators
with compact resolvents. Using the compactness, the convergence
of eigenvalues follow by a straightforward modification
of the proof of Proposition~\ref{Prop.strong}.
In particular, we have the following result for the lowest eigenvalue,
exactly as we would have under the norm-resolvent convergence
described above.
\begin{Corollary}\label{Corol.strong}
Under the hypotheses of Proposition~\ref{Prop.strong},
one has
\begin{equation*}%\label{twisted.infinity}
  \nu_K(\infty) := \lim_{s\to\infty} \nu_K(s) = 3/4
  \,.
\end{equation*}
\end{Corollary}
\begin{proof}
First of all, let us notice that~$\nu_K(s)$ remains bounded as $s \to \infty$.
This is easily seen by the Rayleigh-Ritz variational formula for
the lowest eigenvalue of~$L_s$, in which we use the trial function
of the form $\psi(y) := \varphi(y_1) \mathcal{J}_1(y_2)$,
where $\varphi \in C_0^\infty(\Real)$ is supported outside
$
  \supp(f-1) \supseteq \supp(f_s-1)
$
(\cf~\eqref{consequence.compact}). Indeed,
\begin{equation}\label{ev.bound}
  \nu_K(s) \leq
  \frac{l_s[\psi]}{\,\|\psi\|_{f_s}^2}
  = \frac{\|\dot\varphi\|_{\sii(\Real)}^2
  +\frac{1}{16} \, \|y_1\varphi\|_{\sii(\Real)}^2}
  {\|\varphi\|_{\sii(\Real)}^2}
  \,,
\end{equation}
irrespectively of the properties of~$K$.

Now, let~$\psi_s$ be the positive eigenfunction of~$L_s$
corresponding to~$\nu_K(s)$,
normalized to~$1$ in $\sii_{f_s}(\Omega_0)$.
$\psi_s$ is a solution of the problem~\eqref{re} with $F:=(\nu_K(s)+z)\psi_s$.
It is important that~$F$ is uniformly bounded in~$s$
as an element of $\sii_{f_s}(\Omega_0)$,
due to~\eqref{ev.bound} and the normalization of~$\psi_s$.
Then we can proceed exactly as in the proof of Proposition~\ref{Prop.strong}.

We show that $\{\psi_s\}_{s \geq 0}$ contains a subsequence
$\{\psi_{s_n}\}_{n \in \Nat}$ which is weakly converging
to some~$\psi_\infty$ in~$\mathfrak{D_0}$.
Since~$\mathfrak{D_0}$ is compactly embedded in $\sii(\Omega_0)$,
we know that $\{\psi_{s_n}\}_{n \in \Nat}$ converges to~$\psi_\infty$
strongly in $\sii(\Omega_0)$.
In particular, $\|\psi_\infty\| = 1$,
so that we know that~$\psi_\infty$ is non-trivial.
At the same time, we show that $\psi_\infty \in \mathfrak{H}_1$
and that
$
  \varphi_\infty(y_1)
  := (\mathcal{J}_1,\psi_\infty(y_1,\cdot))_{\sii((-a,a))}
$
vanishes at $y_1=0$.

Taking $v(y) := \varphi(y_1) \mathcal{J}_1(y_2)$
with $\varphi \in C_0^\infty(\Real\!\setminus\!\{0\})$
as the test function in the weak formulation
of the eigenvalue problem~\eqref{re},
with~$s$ being replaced by~$s_n$, and sending~$n$ to infinity,
we eventually find that~$\varphi_\infty$ is an eigenfunction of $h_D+z$
with the eigenvalue $\nu_K(\infty)+z$.
Since~$\psi_\infty$ is obtained as a limit of positive functions,
we know that~$\varphi_\infty$ is positive as well.
Hence, $\nu_K(\infty)$~represents the lowest eigenvalue of~$h_D$.

It remains to recall that the first eigenvalue of~$h_D$ coincides
with the second eigenvalue of~$h$, which is~$3/4$ due to~\eqref{HO.spec}.
\end{proof}
%

%----------------------------------------------%
\subsection{A spectral bound to the decay rate}\label{Sec.improved}
%----------------------------------------------%
%
We come back to~\eqref{spectral.reduction.integral}.
Assume $K=0$ or that there exists a Hardy-type inequality~\eqref{Hardy}.
Recalling~\eqref{untwisted.infinity} and Corollary~\ref{Corol.strong},
we know that for arbitrarily small positive number~$\eps$
there exists a (large) positive time~$s_\eps$ such that
for all $s \geq s_\eps$, we have $\nu_K(s) \geq \nu_K(\infty) - \eps$.
Hence, fixing $\eps>0$, we have
\begin{align*}
  {-\int_0^s \nu_K(r) \, dr}
  &\leq {-\int_0^{s_\eps} \nu_K(r) \, dr} {-[\nu_K(\infty)-\eps](s-{s_\eps})}
  \\
  &\leq \int_0^{s_\eps} |\nu_K(r)| \, dr
  + {[\nu_K(\infty)-\eps] s_\eps} {-[\nu_K(\infty)-\eps] s}
\end{align*}
for all $s \geq s_\eps$.
At the same time, assuming $\eps \leq 1/4$, we trivially have
$$
  {-\int_0^s \nu_K(r) \, dr}
  \leq \int_0^{s_\eps} |\nu_K(r)| \, dr
  + {[\nu_K(\infty)-\eps] s_\eps} {-[\nu_K(\infty)-\eps] s}
$$
also for all $s \leq s_\eps$.
Summing up, for every $s \in [0,\infty)$, we have
\begin{equation}\label{instead}
  \|\tilde{u}(s)\|_{w f_s}
  \leq C_\eps \, e^{-[\nu_K(\infty)-\eps]s} \, \|\tilde{u}_0\|_{w f_0}
  \,,
\end{equation}
where
$
  C_\eps := e^{
  \int_0^{s_\eps} |\nu_K(r)| \, dr + [\nu_K(\infty)-\eps] s_\eps
  }
$.

Now we return to the original variables $(x,t)$ via~\eqref{space-times}.
Using~\eqref{preserve} together with the point-wise estimate $1 \leq w$,
and recalling that $f_0=f$ and  $\tilde{u}_0=u_0$,
it follows from~\eqref{instead} that
$$
  \|u(t)\|_{f}
  = \|\tilde{u}(s)\|_{f_s}
  \leq \|\tilde{u}(s)\|_{w f_s}
  \leq C_\eps \, (1+t)^{-[\nu_K(\infty)-\eps]} \, \|u_0\|_{w f}
$$
for every $t \in [0,\infty)$.
Consequently, we conclude with
$$
  \big\|e^{-(H_K-E_1)t}\big\|_{\sii_{wf}(\Omega_0) \to \sii_f(\Omega_0)}
  = \sup_{u_0 \in \sii_{wf}(\Omega_0)\setminus\{0\}}
  \frac{\|u(t)\|_f}{\ \|u_0\|_{wf}}
  \leq C_\eps \, (1+t)^{-[\mu_\theta(\infty)-\eps]}
$$
for every $t \in [0,\infty)$.
Since~$\eps$ can be made arbitrarily small,
this bound implies
\begin{equation}\label{rate-ev}
  \Gamma_K \geq \nu_K(\infty)
  \,.
\end{equation}
%

%-------------------------------------------%
\subsection{The improved decay rate}\label{Sec.improved.bis}
%-------------------------------------------%
%
Now we arrive to the main result of this paper.
It follows from Proposition~\ref{Prop.straight} that
$\Gamma_0 = 1/4$ (\ie~$K=0$).
The lower bound $\Gamma_0 \geq 1/4$ alternatively follows
from~\eqref{rate-ev} using~\eqref{untwisted.infinity}.
The following theorem states that the decay rate
is three times better in the presence
of a Hardy-type inequality~\eqref{Hardy}.
\begin{Theorem}\label{t:main}
Assume~\eqref{Ass.basic.alt} and~\eqref{Ass.compact}.
If~\eqref{Hardy} holds, then
$$
  \Gamma_K = 3/4
  \,.
$$
\end{Theorem}
\begin{proof}
The assertion $\Gamma_K \geq 3/4$ follows from~\eqref{rate-ev}
using Corollary~\ref{Corol.strong}.
In order to prove the $\Gamma_K \leq 3/4$ it is sufficient to show,
that for some suitable function $\varphi \in C_0^{\infty}(\Omega_0)$,
some constant $c_{\varphi} > 0$ and some constant $t_0 \geq 0$
\begin{equation}\label{e:lowerboundwanted}
\forall t \geq t_0, \quad
\bigl\|e^{-tH_K}\varphi\bigr\|_{L^2_f(\Omega_0)}
\geq c_{\varphi}\,t^{-3/4} e^{E_1 t}.
\end{equation}
Due to \eqref{Ass.compact} the support of $f$ is contained in a rectangle
$\Omega_R:= (-R,R)\times (-a,a)$ for $R>0$.
We choose $\varphi \in C_0^{\infty}(\Omega_0)$
such that $\supp(\varphi) \subset \Omega_0\setminus \Omega_R$.
Recall that $\mathbb{E}_x$ (respectively, $\mathbb{P}_x$)
denote the expectation (respectively, probability measure) corresponding to the
Markov process $(X_t)_{t \geq 0}$ associated to the Dirichlet form~$h_k$.
Define the stopping times $\tau_{ \Omega_0}$ and $\tau_{ \Omega_R}$ by
\begin{displaymath}
  \tau_{ \Omega_0}=\inf\lbrace t \geq 0 \mid X_t \in  \Omega_0\rbrace
  \quad\text{  and   }\quad
  \tau_{\Omega_R}=\inf\lbrace t \geq 0 \mid X_t \in \partial \Omega_R\rbrace.
\end{displaymath}
The process $(X_t)_{0 \leq t < \tau_{ {\Omega_0}}}$
is called Brownian motion on $\Omega_0$ killed at the boundary.
For every $x=(x_1,x_2) \in Q_0$ we then conclude
\begin{equation}\label{e:kill_estim}
\begin{split}
e^{-tH_K}\varphi(x) = \mathbb{E}_x\bigl[\varphi(X_t),;\tau_{ \Omega_0}\bigr] \geq \mathbb{E}_x\bigl[\varphi(X_t),\tau_{\Omega_R}\wedge \tau_{ \Omega_0}>t\bigr],
\end{split}
\end{equation}
where $\tau_{ \Omega_R}\wedge \tau_{ \Omega_0}$ denotes the minimum
of the stopping times $\tau_{\Omega_0}$ and $\tau_{\Omega_R}$.
Integration of \eqref{e:kill_estim}
and using $K  \restriction \Omega_0\setminus\Omega_R = 0$
and hence -- by Lemma \ref{Lem.Taylor} --
$f = 1$ in $\Omega_0 \setminus \Omega_R$ yields
\begin{equation}\label{e:lboundonL2}
\begin{split}
\bigl\|e^{-t H_K}\varphi\bigr\|_{L^2_f(\Omega_0)}^2
&\geq \int_{(R, \infty) \times(-a,a)}\bigl|\mathbb{E}_x
\bigl[\varphi(X_t),\tau_{ \Omega_R}\wedge \tau_{\Omega_0}>t\bigr]
\bigr|^2\,dx \\
&= \int_{(R, \infty) \times(-a,a)}\bigl|\mathbb{E}_x
\bigl[\varphi(X_t),\tau_{ \Omega_{0R}}>t\bigr]\bigr|^2\,dx
\end{split}
\end{equation}
where $\Omega_{0R}:=(R,\infty)\times (-a,a)$ and
\begin{displaymath}
\tau_{\Omega_{0R}}=\inf\lbrace t \geq 0\mid X_t \in \partial\Omega_{0R}\rbrace.
\end{displaymath}
Due to $f = 1$ in $\Omega_0 \setminus \Omega_R$ the stochastic process $(X_t)_{\tau_{\Omega_{0R}}> t \geq 0}$ is a (deterministically time changed by the factor 2) Brownian motion killed, when exiting the set $\Omega_{0R}$. Due to independence of the coordinates we have
\begin{equation}\label{e:killingkernel}
\mathbb{E}_x\bigl[\varphi(X_t),\tau_{\Omega_R}\wedge \tau_{ \Omega_0}>t\bigr] = \sum_{n=1}^{\infty}e^{-E_nt}\mathcal{J}_n(x_1)\int_{\Omega_{0,R}}p_0(t,x_1,y_1)\mathcal{J}_n(_2)\varphi(y_1,y_2)\,dy,
\end{equation}
where
\begin{displaymath}
p_0(t,x,y)
:= \frac{1}{\sqrt{4\pi t}}\bigl(e^{-\frac{(x-y)^2}{4t}}-e^{\frac{(x+y)^2}{4t}}\bigr)
\end{displaymath}
is the transition function of a one-dimensional Brownian motion killed
when hitting $0$. Using \eqref{e:lboundonL2} and \eqref{e:killingkernel}
an elementary calculation gives assertion~\eqref{e:lowerboundwanted}.
\end{proof}
Observe that the proof of Theorem \ref{t:main} demonstrates
that the `transient' effect of negative curvature
on the survival properties of a Brownian particle is as strong as
if we kill a particle when entering the curved region.

%---------------------------------------------%
\subsection{From normwise to pointwise bounds}
%---------------------------------------------%
%
Theorem \ref{t:main} may be reformulated in terms
of certain pointwise assertions.
\begin{Corollary}\label{c:pointwise}

Assume~\eqref{Ass.basic.alt}, ~\eqref{Ass.compact} as well as the \eqref{Hardy}.
Let $x \in \Omega_0$, $\delta > 0$ and a measurable bounded subset
$B \subset \Omega_0$ be given.
Then there exists a constant $C_{B,\delta,x}>0$ such that
\begin{displaymath}
\mathbb{P}_x\bigl(X_t \in B, \tau_{\Omega_0}>t\bigr)
\leq C_{B,\delta,x}\, e^{-E_1 t}\,t^{-\frac{3}{2}+\delta}
\,.
\end{displaymath}
\end{Corollary}
\begin{proof}
We use that according to Proposition \ref{p:fundprop} the integral kernel $e^{-tH_K}(x,y)$ of $e^{-tH_K}$ satisfies
the following Gaussian upper bound
\begin{displaymath}
e^{-tH_K}(x,y) \leq \frac{c_1}{\sqrt{4\pi t}}\, e^{-\frac{(x-y)^2}{4c_2t}}
\end{displaymath}
for some constants $c_1,c_2>0$.
For fixed $x$ set
$$
  \psi_{x,\varepsilon}(y)
  =\frac{c_1}{\sqrt{4\pi \varepsilon}} \, e^{-\frac{(x-y)^2}{4c_2\varepsilon}},
$$
where $\varepsilon$ is chosen small enough such that
$\psi_{x,\varepsilon} \in L^2_{wf}(\Omega_0)$.
Then for $t>\varepsilon $ we have for some constant $C_{\delta}>0$
\begin{equation*}
\begin{split}
\mathbb{P}_x\bigl(X_t \in B, \tau_{\Omega_0}>t\bigr)
&= e^{-tH_K}\chi_B(x) = e^{-\varepsilon H_K}e^{-(t-\varepsilon)H_K}\chi_B(x) \\
&= \bigl(\psi_{x,\varepsilon},e^{-(t-\varepsilon)H_K}\chi_B\bigr)_{f} \\
&= \bigl(e^{-\frac{t-\varepsilon}{2}H_K}\psi_{x,\varepsilon},
e^{-\frac{t-\varepsilon}{2}H_K}\chi_B \bigr)_{f} \\
&\leq \|\psi_{x,\varepsilon}\|_{wf} \, \|\chi_B\|_{wf}
\left[ C_{\delta} \,
\bigl(\frac{t-\varepsilon}{2}\bigr)^{-(\frac{3}{4}-\delta/2)}
e^{-E_1\frac{t-\varepsilon}{2}}\right]^2,
\end{split}
\end{equation*}
where the last inequality follows using the Cauchy-Schwarz inequality
and Theorem \ref{t:main} have been used.
This implies the assertion of the Corollary.
\end{proof}
\begin{Remark}
In the case of positively curved manifolds
satisfying hypotheses~\eqref{Ass.basic.alt}
and~\eqref{Ass.compact}, the decay rate of
$\mathbb{P}_x\bigl(X_t \in B, \tau_{\Omega_0}>t\bigr)$ is exactly exponential,
whereas in the situation of a flat manifold on has $t^{-1/2}e^{-E_1 t}$.

In terms of Tweedie's $R$-theory (see \cite{T74b} and \cite{T74a})
one can therefore conclude that a Brownian particle
in a positively curved tube satisfying condition~\eqref{Ass.basic.alt}
and~\eqref{vanish}
is $E_1$-positive recurrent,
in a flat manifold the Brownian particle is $E_1$-null recurrent
and in the negatively curved tube satisfying \eqref{Ass.basic.alt}
and~\eqref{Ass.compact} the Brownian motion is $E_1$-transient.
\end{Remark}
Let us finally reformulate our findings in the negatively curved case
in another way using conditional probabilities, again.
\begin{Corollary}
Assume~\eqref{Ass.basic.alt} and~\eqref{Ass.compact}. Let $x \in \Omega_0$, $\delta > 0$ and a measurable bounded subset $B \subset \Omega_0$ be given. Then there exists a constant $\tilde{C}_{B,\delta,x}>0$ such that
\begin{displaymath}
\mathbb{P}_x\bigl(X_t \in B \mid \tau_{\Omega_0}>t\bigr)\leq  \tilde{C}_{b,\delta,x}t^{-1+\delta}
\end{displaymath}
\end{Corollary}
\begin{proof}
This follows directly from Corollary \ref{c:pointwise}
together with the fact that for a suitable constant $c_{x}$
\begin{displaymath}
\mathbb{P}_x\bigl(\tau_{\Omega_0}>t\bigr) \geq c_x,.e^{-\lambda_Kt}\,t^{-\frac{1}{2}}.
\end{displaymath}
The latter assertion can be proved by adding a Dirichlet boundary as was done in the proof of Theorem \ref{t:main}.
\end{proof}
%

%--------------------%
\section{Conclusions}\label{Sec.end}
%--------------------%
%
The objective of this paper was to investigate the interplay
between the curvature and the properties of Brownian motion
in the simplest non-trivial case,
when the ambient space is two-dimensional and the motion
in fact quasi-one-dimensional.
More precisely, we were interested in the large time behaviour
of the solution to the heat equation in tubular neighbourhoods
of unbounded geodesics in a two-dimensional Riemannian manifold,
subject to Dirichlet boundary conditions.

Our results are schematically summarized in Table~\ref{table}.
The corresponding precise statements can be found in:
Propositions~\ref{Prop.straight} and~\ref{Prop.critical}
for flat manifolds;
Corollary~\ref{c:cor-pos-curv}
for positively curved manifolds;
and Theorem~\ref{t:main} for negatively curved manifolds.
The moral of the story is that the negative curvature
is `better for travelling', in the sense that the heat semigroup
gains an extra polynomial, geometrically induced decay rate.
The latter is in fact a consequence of the existence
of Hardy-type inequalities in negatively curved manifolds,
which play a central role in our proof. Though the proofs are mainly analytic
some effort has been made in order to connect our findings
with notions and results
available in the probabilistic literature,
\eg\ on Markov chains.

The present paper can be considered as a contribution
to recent works on the consequences
of the existence of Hardy inequality on large-time
behaviour of the heat semigroup
for quantum waveguides
\cite{KZ1,KZ2,GKP,KW}
and magnetic Schr\"odinger operators
\cite{Kovarik_2010,K7}.
More generally, recall that we expect that there is always
an improvement of the decay rate for the heat semigroup
of an operator satisfying a Hardy-type inequality
(\cf~\cite[Conjecture in Sec.~6]{KZ1} and~\cite[Conjecture~1]{FKP}).
The present paper confirms the general conjecture in the particular case
of the Dirichlet Laplace-Beltrami operator in the strip-like surfaces.
As pointed out in the body of the paper,
the Hardy inequality is essentially equivalent to transience properties.
Thus it is reasonable to expect that a combination
of available probabilistic and analytic methods
might be necessary in order to make progress towards a solution
of the above mentioned conjectures.
%\cite[Conjecture in Sec.~6]{KZ1} and~\cite[Conjecture~1]{FKP}.
\medskip \\
\textit{Open Problem:}
One of the characteristic hypotheses of the present paper was that
the curvature~$K$ has compact support.
We expect the same decay rates if this assumption
is replaced by its fast decay at infinity.
However, it is quite possible that a slow decay of curvature
at infinity will improve the decay of the heat semigroup even further.
In particular, can~$\Gamma_K$ be strictly greater than~$3/4$
if~$K$ decays to zero very slowly at infinity?

\subsection*{Acknowledgement}
M.K. would like to thank Laurent Saloff-Coste for clarifying remarks concerning
his work \cite{LSC} and the continuity of transition kernels.
This work has been partially supported by
the Czech Ministry of Education, Youth, and Sports
within the project LC06002 and by the GACR grant No.~P203/11/0701.

%--------------%
% BIBLIOGRAPHY
%--------------%
%
\addcontentsline{toc}{section}{References}
%\newpage
%\bibliography{bib}

\begin{thebibliography}{10}

\bibitem{Brezis_FR}
H.~Br{\'e}zis, \emph{{Analyse fonctionnelle: Th\'eorie et applications}},
  Dunod, 2002.

\bibitem{Davies}
E.~B. Davies, \emph{Spectral theory and differential operators}, Camb. Univ
  Press, Cambridge, 1995.

\bibitem{Dav}
E. B. Davies, \emph{$L^p$ spectral independence and $L^1$ analyticity},
J. London Math. Soc. \textbf{52} (1995), 177--184

\bibitem{DDI}
Y.~Dermenjian, M.~Durand, and V.~Iftimie, \emph{Spectral analysis of an
  acoustic multistratified perturbed cylinder}, Commun. in Partial Differential
  Equations \textbf{23} (1998), no.~1{\&}2, 141--169.

\bibitem{Duro-Zuazua_1999}
G.~Duro and E.~Zuazua, \emph{Large time behavior for convection-diffusion
  equations in {$\mathbb{R}^N$} with asymptotically constant diffusion},
  Commun. in Partial Differential Equations \textbf{24} (1999), 1283--1340.

\bibitem{EKK}
T.~Ekholm, H.~Kova{\v{r}}{\'\i}k, and D.~Krej\v{c}i\v{r}\'{\i}k, \emph{A
  {H}ardy inequality in twisted waveguides}, Arch. Ration. Mech. Anal.
  \textbf{188} (2008), 245--264.

\bibitem{Escobedo-Kavian_1987}
M.~Escobedo and O.~Kavian, \emph{Variational problems related to self-similar
  solutions of the heat equation}, Nonlinear Anal.-Theor. \textbf{11} (1987),
  1103--1133.

\bibitem{Escobedo-Zuazua_1991}
M.~Escobedo and E.~Zuazua, \emph{Large time behavior for convection-diffusion
  equations in {$R^N$}}, J. Funct. Anal. \textbf{100} (1991), 119--161.

\bibitem{FKP}
M.~Frass, D.~Krej\v{c}i\v{r}\'{\i}k, and Y.~Pinchover, \emph{On some strong
  ratio limit theorems for heat kernels}, Discrete Contin. Dynam. Systems A
  \textbf{28} (2010), 495--509.

\bibitem{Gray}
A.~Gray, \emph{Tubes}, Addison-Wesley Publishing Company, New York, 1990.

\bibitem{Griffiths}
D.~J. Griffiths, \emph{Introduction to quantum mechanics}, Prentice Hall, Upper
  Saddle River, NJ, 1995.

\bibitem{Grig}
A. Grigoryan, \emph{Analytic and geometric background of recurrence and non-explosion of the Brownian motion on Riemannian manifolds},
Bull. Amer. Math. Soc. (N.S.) \textbf{36} (1999), 135--249.

\bibitem{GSL} A. Grigor'yan and L. Saloff-Coste, \emph{Hitting probabilities for Brownian motion on Riemannian manifolds}, J. Math. Pures Appl. \textbf{81} (2002), no. 2, 115--142.

\bibitem{GKP} G.~Grillo, H.~Kova\v{r}\'ik and Y.~Pinchover,
\emph{Sharp two-sided heat kernel estimates of twisted tubes and applications},
submitted; preprint on arXiv:1105.0842v1 [math.AP] (2011).

\bibitem{Hartman_1964}
P.~Hartman, \emph{Geodesic parallel coordinates in the large}, American Journal
  of Mathematics \textbf{86} (1964), 705--727.

\bibitem{HV} R. Hempel and J. Voigt,
\emph{On the $L_p$-spectrum of Schr\"odinger operators}, J. Math. Anal. Appl. \textbf{121} (1987), 138--159

\bibitem{Hsu} E. P. Hsu and G. Qin, \emph{Volume growth and escape rate of Brownian motion on a complete manifold}, Ann. Probab. \textbf{38} (2010), 1570--1582

\bibitem{Kato}
T.~Kato, \emph{Perturbation theory for linear operators}, Springer-Verlag,
  Berlin, 1966.

\bibitem{Kendall1} W. S. Kendall, \emph{Brownian motion, negative curvature, and harmonic maps}, Stochastic integrals ({P}roc. {S}ympos., {U}niv. {D}urham, {D}urham, 1980) {Lecture Notes in Math.}, {851}, {479--491}, {Springer}, {Berlin}, {1981},

\bibitem{Kendall2} W. S. Kendall, \emph{Stochastic differential geometry: an introduction}. Acta Appl. Math. \textbf{9} (1987), 29--60.

\bibitem{KW} M. Kolb and A. W\"ubker, \emph{Brownian motion in twisted Dirichlet-Neumann tubes}, preprint (2011), submitted

\bibitem{Kovarik_2010}
H.~Kova\v{r}{\'\i}k, \emph{{Heat kernels of two dimensional magnetic
  Schr\"odinger and Pauli operators}},
  Calc. Var. Partial Differential Equations, to appear;
  preprint on arXiv:1007.1851v1 [math-ph] (2010).

\bibitem{Krej1}
D.~Krej\v{c}i\v{r}\'{\i}k, \emph{Quantum strips on surfaces}, J.~Geom. Phys.
  \textbf{45} (2003), no.~1--2, 203--217.

\bibitem{K3}
D.~Krej\v{c}i\v{r}\'{\i}k,
\emph{Hardy inequalities in strips on ruled surfaces}, J. Inequal.
  Appl. \textbf{2006} (2006), Article ID 46409, 10 pages.

\bibitem{K6-with-erratum}
D.~Krej\v{c}i\v{r}\'{\i}k, \emph{Twisting versus bending in quantum
  waveguides}, Analysis on Graphs and its Applications, Cambridge, 2007
  (P.~Exner et~al., ed.), Proc. Sympos. Pure Math., vol.~77, Amer. Math. Soc.,
  Providence, RI, 2008, pp.~617--636. See arXiv:0712.3371v2 [math--ph] (2009)
  for a corrected version.

\bibitem{K7}
D.~Krej\v{c}i\v{r}\'{\i}k, \emph{{The improved decay rate for the heat
  semigroup with local magnetic field in the plane}}, submitted; preprint on
  arXiv:1101.1806 [math.AP] (2011).

\bibitem{KKriz}
D.~Krej\v{c}i\v{r}\'{\i}k and J.~K\v{r}\'{\i}\v{z}, \emph{On the spectrum of
  curved quantum waveguides}, Publ.~RIMS, Kyoto University \textbf{41} (2005),
  no.~3, 757--791.

\bibitem{KZ1}
D.~Krej\v{c}i\v{r}\'{\i}k and E.~Zuazua, \emph{The {H}ardy inequality and the
  heat equation in twisted tubes}, J. Math. Pures Appl. \textbf{94} (2010),
  277--303.

\bibitem{KZ2}
D.~Krej\v{c}i\v{r}\'{\i}k and E.~Zuazua,
\emph{The asymptotic behaviour of the heat equation in a twisted
  {D}irichlet-{N}eumann waveguide}, J. Differential Equations \textbf{250}
  (2011), 2334--2346.

\bibitem{McGO} I. McGillivray and E. M. Ouhabaz, \emph{Some spectral properties of recurrent semigroups}, Arch. Math. \textbf{66} (1996), 233--242

\bibitem{Ne} R. Neel, \emph{The small time asymptotics of the heat kernel at the cut locus}, Comm. Anal. Geom. \textbf{15} (2007), 845--890

\bibitem{N}
J. R. Norris, \emph{Long-Time Behavior of Heat Flow: Global Estimates and exact Asymptotics}, Arch. Rational Mech. Anal.
\textbf{140} (1997), 161--195

\bibitem{Ouh} E.-M. Ouhabaz, \emph{Invariance of closed convex sets and domination criteria for semigroups}, Potential Anal. \textbf{5} (1996), 611--625.

\bibitem{Pinchover_2004}
Y.~Pinchover, \emph{Large time behavior of the heat kernel},
J. Funct. Anal. \textbf{206} (2004), 191--209.

\bibitem{Pinchover_2007}
Y.~Pinchover, \emph{Topics in the theory of positive solutions of second-order
  elliptic and parabolic partial differential equations}, Spectral Theory and
  Mathematical Physics: A Festschrift in Honor of Barry Simon's 60th Birthday
  ({F.~Gesztesy, et al.}, ed.), Proc. Sympos. Pure Math., vol.~76, Amer. Math.
  Soc., Providence, RI, 2007, pp.~329--356.

\bibitem{Pin}  R. G. Pinsky, \emph{Positive harmonic functions and diffusion}, Cambridge Studies in Advanced Mathematics, 45. Cambridge University Press, Cambridge, 1995. xvi+474 pp.

\bibitem{LSC} L. Saloff-Coste,
\emph{Uniformly elliptic operators on Riemannian manifolds},
J. Differential Geometry, \textbf{36} (1992), 417--450

\bibitem{Sim3} B. Simon,
\emph{ Brownian motion, $L^{p}$ properties of Schr\"odinger operators
and the localization of binding},
J. Funct. Anal. \textbf{35} (1980), 215--229.

\bibitem{Sim2} B. Simon, \emph{Large Time Behavior of the $L^p$ Norm of
Schr\"odinger Semigroups}, J. Funct. Anal. \textbf{40} (1981), 66--83

\bibitem{Sim} B. Simon,
\emph{Large time behavior of the heat kernel: on a theorem of Chavel and Karp},
Proc. Amer. Math. Soc. \textbf{118} (1993), 513--514.

\bibitem{Spivak4}
M.~Spivak, \emph{A comprehensive introduction to differential geometry},
  vol.~IV, Publish or Perish, Boston, Mass., 1975.

\bibitem{StI}
K. T. Sturm, \emph{Analysis on local Dirichlet Spaces -- I. Recurrence, conservativeness and $L^p$-Liouville}, J. reine angew. Math. \textbf{456} (1994), 173--196

\bibitem{StII}
K. T. Sturm, \emph{Analysis on local Dirichlet Spaces -- II. Upper Gaussian Estimats for the Fundamential Solutions of Parabolic Equations}, Osaka J. Math \textbf{32} (1995), 275--312

\bibitem{TT} P. Tuominen and R. L. Tweedie, \emph{Exponential Decay and Ergodicity of General Markov
Processes and Their Discrete Skeletons}, Adv. in Appl. Probab. \textbf{11} (1979), 784--803

\bibitem{T74b} R. L. Tweedie, $R$-theory for Markov chains on a general state space. II. $R$-subinvariant measures for $R$-transient chains. \textit{Ann. Probability} \textbf{2} (1974), 865--878.

\bibitem{T74a} R. L. Tweedie, $R$-theory for Markov chains on a general state space. I. Solidarity properties and $R$-recurrent chains. \textit{Ann. Probability} \textbf{2} (1974), 840--864.

\bibitem{Weid} J. Weidmann, \emph{Lineare Operatoren im Hilbertraum: Teil 1 Grundlagen}, Teubner, Stuttgart, 2000
\end{thebibliography}
%\bibliographystyle{amsplain}
%
\providecommand{\bysame}{\leavevmode\hbox to3em{\hrulefill}\thinspace}
\providecommand{\MR}{\relax\ifhmode\unskip\space\fi MR }
\providecommand{\MRhref}[2]{%
  \href{http://www.ams.org/mathscinet-getitem?mr=#1}{#2}
}
\providecommand{\href}[2]{#2}

\end{document}